\documentclass[12pt,english]{amsart}
\usepackage[T1]{fontenc}
\usepackage[latin9]{inputenc}
\usepackage{xcolor}
\usepackage{babel}
\usepackage{units}
\usepackage{mathtools}
\usepackage{amstext}
\usepackage{amsthm}
\usepackage{amssymb}
\usepackage{graphicx}
\usepackage{geometry}
\PassOptionsToPackage{normalem}{ulem}
\usepackage{ulem}
\usepackage[pdfusetitle,
 bookmarks=true,bookmarksnumbered=false,bookmarksopen=false,
 breaklinks=false,pdfborder={0 0 1},backref=false,colorlinks=false]
 {hyperref}

\makeatletter

\providecolor{lyxadded}{rgb}{0,0,1}
\providecolor{lyxdeleted}{rgb}{1,0,0}
\DeclareRobustCommand{\mklyxadded}[1]{\textcolor{lyxadded}\bgroup#1\egroup}
\DeclareRobustCommand{\mklyxdeleted}[1]{\textcolor{lyxdeleted}\bgroup\mklyxsout{#1}\egroup}
\DeclareRobustCommand{\mklyxsout}[1]{\ifx\\#1\else\sout{#1}\fi}

\numberwithin{equation}{section}
\numberwithin{figure}{section}

\usepackage{babel}
\usepackage{color}
\providecommand{\definitionname}{Definition}
\providecommand{\examplename}{Example}
\providecommand{\lemmaname}{Lemma}
\providecommand{\propositionname}{Proposition}
\providecommand{\remarkname}{Remark}
\providecommand{\theoremname}{Theorem}
\definecolor{green}{RGB}{0, 180, 0}
\definecolor{cyan}{RGB}{0, 180, 180}
\definecolor{yellow}{RGB}{211,211,0}
\usepackage{upgreek}
\newcommand{\comment}[1]{}

\providecommand{\corollaryname}{Corollary}

\makeatother

\providecommand\theoremname{Theorem}
\theoremstyle{plain}
\newtheorem{thm}{\protect\theoremname}[section]
\providecommand\definitionname{Definition}
\theoremstyle{definition}
\newtheorem{defn}[thm]{\protect\definitionname}
\providecommand\examplename{Example}
\newtheorem{example}[thm]{\protect\examplename}
\providecommand\assumptionname{Assumption}
\theoremstyle{plain}
\newtheorem{assumption}[thm]{\protect\assumptionname}
\providecommand\remarkname{}
\theoremstyle{remark}
\newtheorem*{rem*}{\protect\remarkname}
\newtheorem{rem}[thm]{\protect\remarkname}
\providecommand\corollaryname{Corollary}
\theoremstyle{plain}
\newtheorem{cor}[thm]{\protect\corollaryname}
\providecommand\propositionname{Proposition}
\newtheorem{prop}[thm]{\protect\propositionname}
\providecommand\lemmaname{Lemma}
\newtheorem{lem}[thm]{\protect\lemmaname}
\begin{document}

\global\long\def\rmi{\mathbf{\textrm{i}}}%
\global\long\def\rme{\mathbf{\textrm{e}}}%
\global\long\def\rmd{\mathbf{\textrm{d}}}%
\global\long\def\id{\mathbf{1}}%
\global\long\def\d{\partial}%
\global\long\def\Z{\mathbb{Z}}%
\global\long\def\R{\mathbb{R}}%
\global\long\def\T{\mathbb{T}}%
\global\long\def\C{\mathbb{C}}%
\global\long\def\N{\mathbb{N}}%
\global\long\def\Q{\mathbb{Q}}%
\global\long\def\B{\mathcal{B}}%
\global\long\def\F{\mathcal{F}}%
\global\long\def\D{\mathcal{D}}%
\global\long\def\set#1#2{\left\{  #1~:~#2\right\}  }%
\global\long\def\undercom#1#2{\underset{_{#2}}{\underbrace{#1}}}%
\global\long\def\spec#1{\mathrm{Spec}\left(#1\right)}%
\global\long\def\br#1{\left(#1\right)}%
\global\long\def\bra#1{\left\langle #1\right\rangle }%


\global\long\def\E{\mathcal{E}}%
\global\long\def\Ev{\mathcal{E}_{v}}%
\global\long\def\V{\mathcal{V}}%

\global\long\def\P{P}%

\global\long\def\lmin{\mathbf{\textrm{\ensuremath{\ell_{\text{min}}}}}}%
\global\long\def\vp{v_{+}}%
\global\long\def\vm{v_{-}}%
\global\long\def\vpm{v_{\pm}}%
\global\long\def\del{\delta_{\alpha}}%
\global\long\def\A{\mathcal{A}}%

\global\long\def\a{\bar{a}}%

\global\long\def\sft{\mathrm{S}}%

\global\long\def\Ta{T_{a}}%

\global\long\def\T#1{\mathrm{T}_{#1}}%

\global\long\def\TG{\Gamma^{(n)}_{\omega}}%

\global\long\def\Gma{\Gamma_{a}}%
\global\long\def\va{v_{a}}%
\global\long\def\fe{f_{E}}%

\global\long\def\mnw{\mathcal{M}_{n}\left(\omega,E\right)}%
\global\long\def\mw{\mathcal{M}\left(\omega,E\right)}%
\global\long\def\No#1{N_{\Omega}\left(#1\right)}%
\global\long\def\ct#1{\#_{\omega}a\left(#1\right)}%

\global\long\def\mtw#1{\mathcal{M}\left(#1\omega,E\right)}%
\global\long\def\UH{\mathcal{UH}}%
\global\long\def\NUH{\mathcal{NUH}}%

\title{Zero measure Cantor spectrum\\
 for aperiodic metric tiling graphs}
\author{Ram Band \and Gilad Sofer}
\address{Department of Mathematics\\
Technion - Israel Institute of Technology\\
Haifa, Israel}
\begin{abstract}
We study a family of aperiodic metric graphs generated by one-dimensional
tilings. We prove that, generically, the spectrum of the associated
Kirchhoff Laplacian has Lebesgue measure zero and, apart from a discrete
set of isolated eigenvalues, is a generalized Cantor set. For a large
subclass of these graphs, we further show that isolated eigenvalues
are generically absent. This extends several well-known results from
the setting of discrete ergodic Schrödinger operators on $\Z$ to
a large family of metric graphs. Earlier results of this kind for
metric graphs typically assume equilateral edges and exploit symmetries
that reduce the analysis to a Sturm-Liouville operator on the real
line. By contrast, the graph families considered here need not have
equilateral edges or the symmetries needed for such a reduction and
therefore require a different analytical approach.
\end{abstract}

\maketitle

\section{Introduction and main results \label{sec: Intro and main results}}

This paper studies zero measure Cantor spectrum for aperiodic metric
graphs whose geometry is generated by one-dimensional symbolic dynamics.

The spectral theory of aperiodic Schrödinger operators has been a
central topic in mathematical physics since the 1970s. A widely studied
class of operators is discrete ergodic Schrödinger operators on $\Z$,
where even simple aperiodic models display nontrivial spectral properties.
One interesting feature shared by many aperiodic models is having
Cantor spectrum of zero Lebesgue measure and purely singular continuous
spectral measures. This occurs, for instance, for Sturmian Hamiltonians
\cite{Bellissard1989,Suetoe_cmp87} and other classes of low-complexity
discrete Schrödinger operators \cite{MR2229664,lima2003uniform,damanik2003uniform,bovier1993spectral,lenz2002singular}.
Analogous results are also known for ergodic continuum Schrödinger
operators on $\R$ \cite{DamFilGor_ahp15,LenSeiSto_jde14}.

The goal of this work is to study analogous results on networks, where
the aperiodicity manifests through the geometric structure of the
underlying space itself, rather than via a potential. The objects
we study here are continuum Schrödinger operators on metric (quantum)
graphs inspired by one-dimensional tilings, which we call \textit{tiling
graphs,} that provide a natural generalization of the usual models
on $\R$. An important work in this direction is \cite{DamFanSuk_ahp20},
where zero measure spectrum was proven for equilateral metric antitrees.\textit{
}Other related works include results on generic singular continuous
spectrum for radially symmetric metric trees \cite{Exner2014,breuer2009singular,Rohleder2017},
as well as Anderson localization for random radial tree models \cite{MR4081117}
and for quantum graphs generated by subshifts of finite type \cite{Safronov_amp2026}.
A common feature to these works is the presence of symmetries, which
effectively reduce the graph operator to a Sturm-Liouville operator
on $\R$. Our goal is to go beyond that, by studying aperiodic graphs
that do not necessarily possess an immediate symmetry which allows
for such a reduction. This could lay the foundation for further spectral
analysis of graphs with a more complicated geometric structure.

\begin{figure}
\includegraphics[scale=0.75]{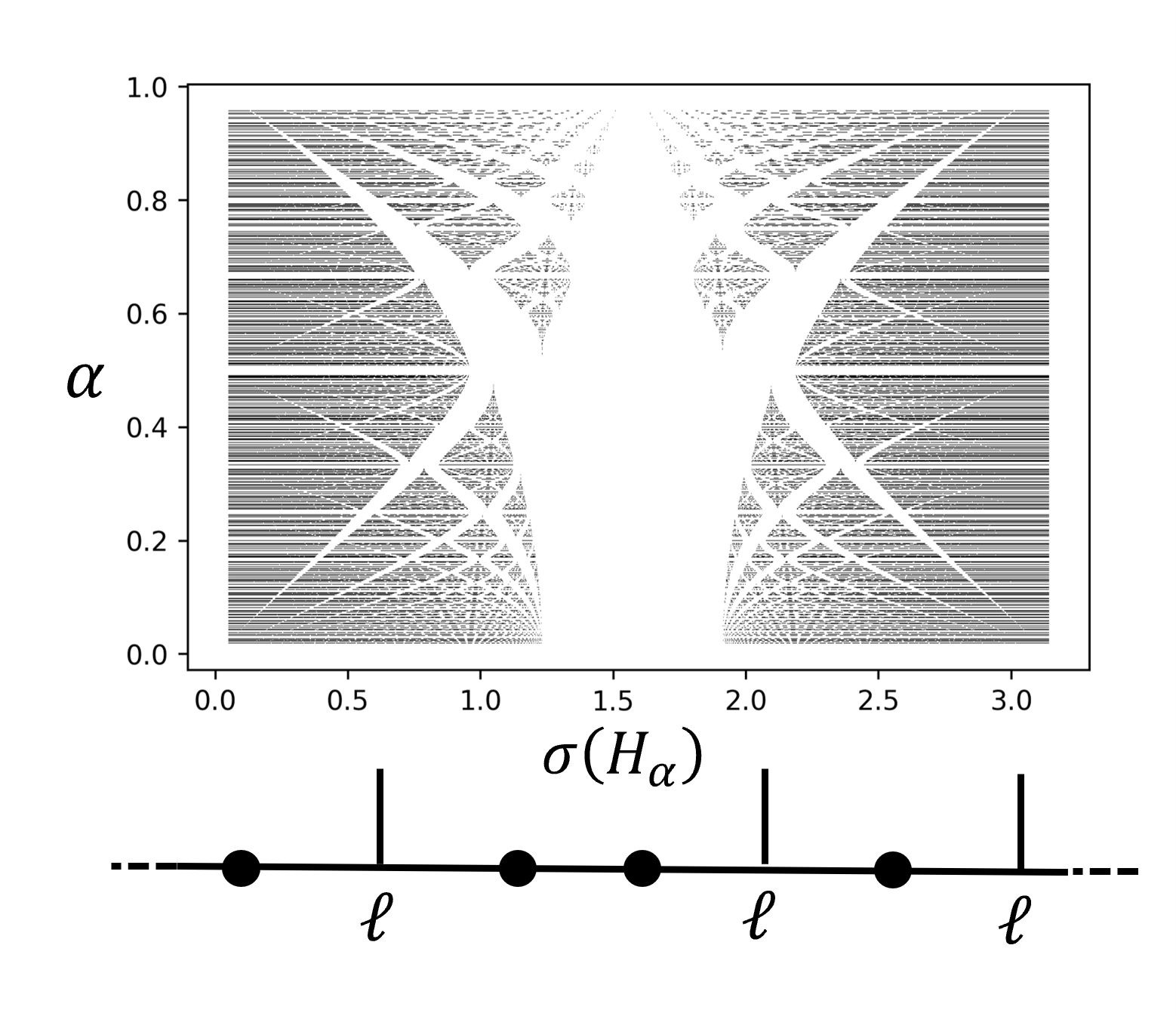}

\caption{The \textquotedblleft Kohmoto butterfly\textquotedblright{} for Sturmian
comb graphs -- parts of $\sigma\left(H_{\alpha}\right)$ as a function
of $\alpha$. The horizontal axis represents $k\in\protect\R$ such
that $k^{2}\in\sigma\left(H_{\alpha}\right)$. \label{fig: butterfly-1}}
\end{figure}

The graphs we study in this work are constructed as one-dimensional
tilings which are generated by minimal, aperiodic dynamical systems,
satisfying Boshernitzan's condition. Under an explicit and Baire-generic
assumption on the tiles forming the graph, we prove that the spectrum
of these graphs is of zero Lebesgue measure (Theorem \ref{thm:TMP}),
and in particular has no absolutely continuous part. After removing
a possible discrete set of isolated eigenvalues, the remaining spectrum
is shown to form a generalized Cantor set.

For a large class of tiling graphs, called decorated $\Z$-graphs,
the result can be strengthened. For a Baire-generic choice of the
edge lengths, the exceptional isolated eigenvalues are absent, and
the spectrum itself is a zero measure generalized Cantor set (Theorem
\ref{thm:Cantor}). In this setting, the exceptional energies are
associated with compactly supported eigenfunctions, which can be excluded
generically by perturbing the edge lengths. Figure \ref{fig: butterfly-1}
illustrates these results for Sturmian comb graphs (see Example~\ref{exa: Sturm-comb}
and Section~\ref{sec:Discussion-and-Example}). These graphs depend
on the Sturmian frequency, $\alpha$. Whenever this frequency is irrational,
the spectrum of the graph is a zero measure Cantor set, giving rise
to an intricate fractal pattern which resembles the Kohmoto butterfly
\cite{KohKadTan_prl83,OstKim_phys85}.

A key ingredient in the proof is a Borg-Marchenko type theorem for
tiling graphs (Theorem \ref{thm:BM}), showing that the underlying
aperiodic tiling can be reconstructed from the Weyl-Titchmarsh $m$-function
of the graph. Our proof is less standard than usual one-dimensional
Borg-Marchenko results, and relies on ideas from quantum chaos on
graphs. Namely, we use wave propagation along the graph to derive
a periodic-orbit expansion of the $m$-function, from which the successive
tiles can be recovered. This inverse-spectral result relates the graph
geometry to the more standard tools for ergodic one-dimensional operators,
namely Kotani theory.

The paper is structured as follows: after presenting the necessary
background and definitions in the next subsections, we state our main
results. Section \ref{sec:BM theorem} proves the Borg-Marchenko theorem
and develops the periodic-orbit expansion of the $m$-function. Section
\ref{sec:Kotani} introduces the transfer matrix and Lyapunov exponent
and uses those to give a dynamical characterization of the spectrum.
Section \ref{sec:Cantor spectrum proof} combines this with Kotani
theory and an oracle theorem to prove Theorems \ref{thm:TMP} and
\ref{thm:Cantor}. Lastly, Section \ref{sec:Discussion-and-Example}
discusses the particular family of Sturmian comb graphs.

\subsection{Quantum graphs\label{subsec:Quantum-graph-preliminaries}}

A \emph{metric graph} is a pair $\Gamma=(G,\vec{\ell})$, where $G=(\mathcal{V},\mathcal{E})$
is a combinatorial graph with vertex set $\mathcal{V}$ and edge set
$\mathcal{E}$, and $\vec{\ell}\in\mathbb{R}^{\left|\E\right|}_{+}$
is a vector assigning a positive length to each edge in $\mathcal{E}$.
Each edge $e\in\E$ is identified with the interval $\left[0,\ell_{e}\right]$,
which equips $\Gamma$ with the natural structure of a metric space.

For a vertex $v$, we write $e\sim v$ to indicate that the edge $e$
is incident to $v$. The degree of a vertex $v$, denoted by $\deg(v)$,
is defined as the number of edges incident to $v$, with loop edges
(edges that connect $v$ to itself) counted twice.

A \emph{quantum graph} is a metric graph $\Gamma$ equipped with a
self-adjoint differential operator $H$ acting on the Sobolev space
$H^{2}\left(\Gamma\right):=\oplus_{e\in\mathcal{E}}H^{2}\left(0,\ell_{e}\right)$.
In this work, the operator acts as the Laplacian on each edge $H=-\frac{d^{2}}{dx^{2}}$,
together with the \textit{Neumann-Kirchhoff} boundary conditions imposed
at each vertex:
\begin{align}
 & f|_{e}\left(v\right)=f|_{e'}\left(v\right),\ensuremath{\quad}\ensuremath{\forall e,e'\sim v},\label{eq:-15-2}\\
 & \sum_{e\sim v}f'|_{e}\left(v\right)=0,\label{eq:-16-2}
\end{align}
where the derivatives are taken in the outwards pointing direction
from the vertex. These vertex conditions are also known as Kirchhoff,
continuity-Kirchhoff, or standard vertex conditions. For an extensive
introduction to quantum graphs, see \cite{GnuSmi_ap06,BanGnu_qg-exerices18,BerKuc_graphs,Berkolaiko_qg-intro17,Kurasov_Book,KosNic_book2023}.

In this work, $\Gamma$ will be an infinite, connected metric graph,
where the vertex degrees $\deg\left(v\right)$ are uniformly bounded
from above, and the edge lengths $\ell_{e}$ are uniformly bounded
from above and also uniformly bounded from below by some $\ell_{0}>0$.
Under these conditions, the associated Kirchhoff Laplacian is self-adjoint
and non-negative, see \cite[thm 1.4.19]{BerKuc_graphs}.

\subsection{Dynamics\label{subsec:Dynamics}}

The aperiodic metric graphs we consider in this work are defined through
one-dimensional dynamical systems that govern their geometric structure.
We now introduce the relevant definitions.

Let $\mathcal{A}$ be a finite set, which we call an alphabet. We
consider the space of bi-infinite sequences $\mathcal{A}^{\mathbb{Z}}$,
equipped with the product topology, as induced by the following metric:
\begin{equation}
d\left(\omega,\omega'\right)=\sum_{n\in\Z}\frac{1-\delta_{\omega\left(n\right),\omega'\left(n\right)}}{2^{\left|n\right|}},\label{eq:disc-metric}
\end{equation}
where $\delta_{i,j}$ is the Kronecker delta. The space $\mathcal{A}^{\mathbb{Z}}$
is naturally equipped with the \textit{shift} map:
\begin{align}
 & \sft:\mathcal{A}^{\mathbb{Z}}\rightarrow\mathcal{A}^{\mathbb{Z}},\label{eq:shift}\\
 & \sft\omega\left(n\right)=\omega\left(n+1\right).\label{eq:shift2}
\end{align}

\begin{defn}
\label{def:Subshift}A \textit{subshift} is a closed, shift-invariant
subset $\Omega\subset\mathcal{A}^{\mathbb{Z}}$. A subshift is called
\textit{minimal} if it contains no proper nontrivial subshifts, and
is called \textit{aperiodic }if it contains no periodic elements,
i.e., an element $\omega$ such that $\sft^{p}\omega=\omega$ for
some $p\neq0$.
\end{defn}

Given a finite subword $W=W\left(1\right)...W\left(k\right)$ of an
element in $\Omega$, we define the corresponding \textit{cylinder
set} by
\begin{equation}
V_{W}=\left\{ \omega\in\Omega:\omega|_{\left[0,k-1\right]}=W\right\} .\label{eq:bosh1}
\end{equation}
If $\mu$ is a $\sft$-invariant probability measure on $\Omega$,
we define the quantity
\begin{equation}
\eta_{\mu}\left(n\right)=\min\left\{ \mu\left(V_{W}\right):\left|W\right|=n\right\} .\label{eq:bosh2}
\end{equation}
A minimal subshift $\Omega$ is said to satisfy \textit{Boshernitzan's
condition} if there exists an ergodic $S$-invariant probability measure
$\mu$ on $\Omega$ such that
\begin{equation}
\limsup_{n\rightarrow\infty}n\cdot\eta_{\mu}\left(n\right)>0.\label{eq:Boshernitzan}
\end{equation}
This condition, introduced by Boshernitzan in \cite{MR808101,MR801288},
is a quantitative recurrence condition for an aperiodic subshift.
It holds for a wide variety of subshifts, as described in \cite{MR2229664}.
For minimal subshifts, Boshernitzan\textquoteright s condition implies
uniformity for locally constant cocycles, as discussed in Section
\ref{sec:Kotani}. This is the main role of this condition in our
analysis.
\begin{example}
\label{exa:Sturmian}Let $\alpha\in\left(0,1\right)\backslash\mathbb{Q}$.
For $\theta\in[0,1)$, we define the Sturmian sequence $\left(\omega_{\alpha,\theta}\left(n\right)\right)_{n\in\Z}$
on the alphabet $\A=\{0,1\}$ by
\begin{equation}
\omega_{\alpha,\theta}\left(n\right)=\chi_{(1-\alpha,1]}\left(n\alpha+\theta\text{ mod \ensuremath{1}}\right).\label{eq:sturm}
\end{equation}
 The Sturmian subshift is then defined as
\begin{equation}
\Omega_{\alpha}=\overline{\left\{ \omega_{\alpha,\theta}:\theta\in[0,1)\right\} }.\label{eq:Sturm-subshift}
\end{equation}
$\Omega_{\alpha}$ is a minimal, aperiodic subshift over the alphabet
$\left\{ 0,1\right\} $, and it satisfies Boshernitzan's condition.
The most famous example is the Fibonacci subshift, obtained for $\alpha=\frac{\sqrt{5}-1}{2}$
(see \cite{DamGorYes_inv16}).
\end{example}

\subsection{Tiling graphs\label{subsec:Metric-tiling-graphs}}

We now introduce the class of \textit{aperiodic tiling graphs}, which
are the center of this paper. To define those, we fix a minimal, aperiodic
subshift $\Omega$ over a finite alphabet $\mathcal{A}$, satisfying
Boshernitzan's condition.

Let $\left(\Ta\right)_{a\in\mathcal{A}}$ be a family of compact metric
graphs, which we call the \textit{tiles}. For each $\Ta$, we select
two distinct vertices $v^{a}_{1},v^{a}_{2}\in\V\left(\Ta\right)$.
Given $\omega\in\Omega$, construct an infinite metric graph $\Gamma_{\omega}$
in the following way. Form a chain of graphs $\left(\TG\right)_{n\in\Z}$,
such that
\begin{equation}
\TG=\T{\omega(n)},\label{eq:tiling-n}
\end{equation}
and the graphs $\TG,\Gamma^{\left(n+1\right)}_{\omega}$ are connected
to each other by identifying the vertices $v^{\omega\left(n\right)}_{2}\in\V\left(\TG\right)$
and $v^{\omega\left(n+1\right)}_{1}\in\V\left(\Gamma^{\left(n+1\right)}_{\omega}\right)$
(see Figure \ref{fig: TilingGraphs}). This results in an infinite
graph which we denote by $\Gamma_{\omega}$. Associate to each such
graph an origin $o\left(\Gamma_{\omega}\right)\in\Gamma_{\omega}$,
which is taken to be the vertex $v^{\omega\left(0\right)}_{1}\in\V\left(\Gamma^{\left(0\right)}_{\omega}\right)$.

\begin{figure}
\includegraphics[scale=0.5]{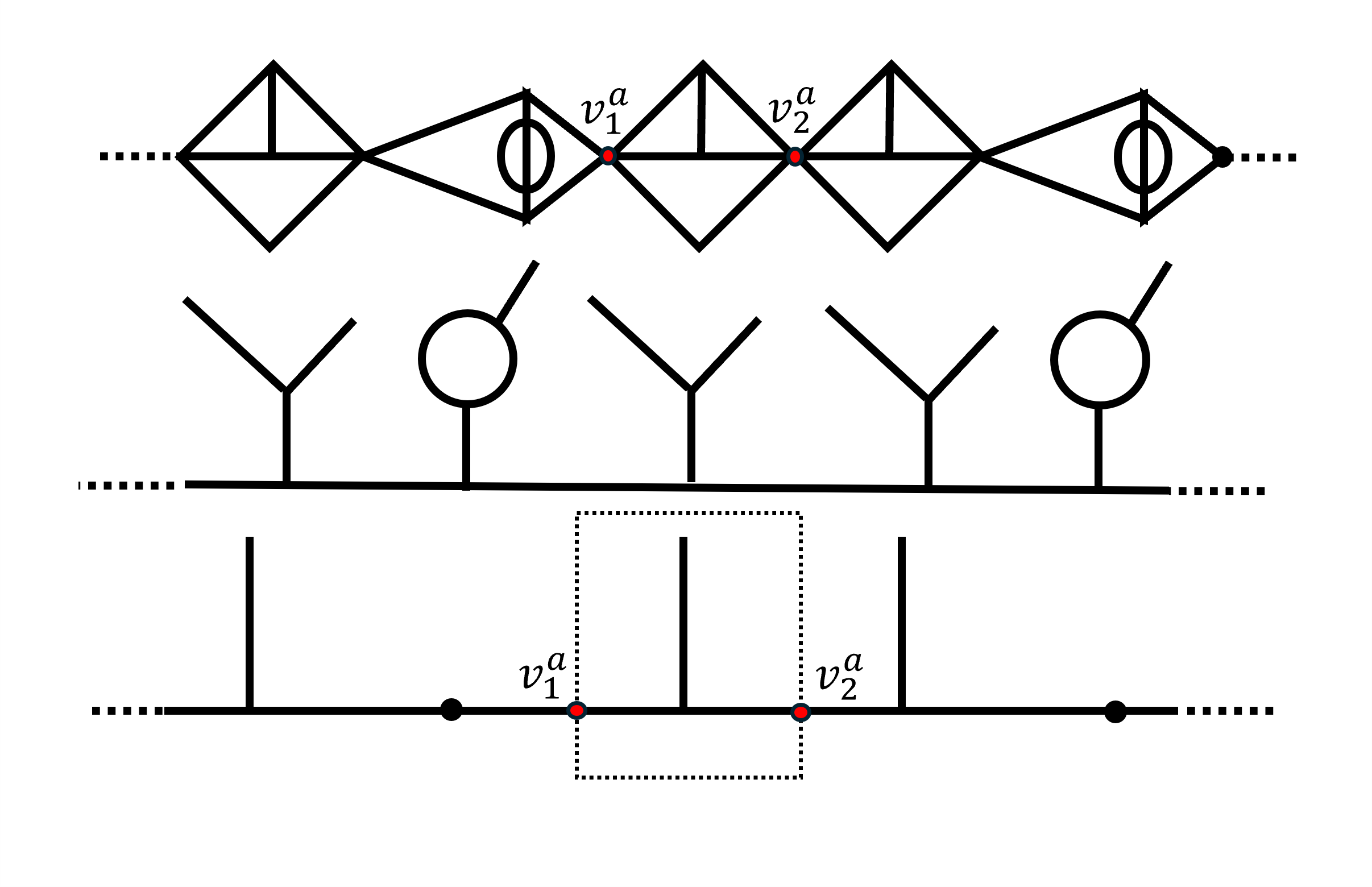}

\caption{Three examples of tiling graphs, where the bottom two graphs are decorated
$\protect\Z$-graphs. For the top and bottom graph the attachment
vertices $v^{a}_{1}$ and $v^{a}_{2}$ of the central tile are marked
in red.\label{fig: TilingGraphs}}
\end{figure}

One important family of tiling graphs on which we also focus is decorated
$\Z$-graphs:
\begin{defn}[Decorated $\Z$-graph]
\label{def: decorated lines} Let $\left\{ \left(\D_{a},v_{a}\right)\right\} _{a\in\A}$
be a set of compact graphs (called \textit{decorations}) with marked
vertices $v_{a}\in\D_{a}$ (called the \textit{base vertices}). Given
a subshift $\Omega$ over the alphabet $\A$, we can naturally construct
a tiling graph $\Gamma_{\omega}$ by taking the line graph over $\mathbb{Z}$,
and for each $a\in\A$, attaching a copy of $\D_{a}$ through the
base vertex $v_{a}$ at all integers $n$ such that $\omega\left(n\right)=a$.
We call such a $\Gamma_{\omega}$ a decorated $\Z$-graph. It is indeed
a particular case of a tiling graph, where the tiles $\left(\Ta\right)_{a\in\A}$
are the graphs $\left(\D_{a}\right)_{a\in\A}$, with two dangling
edges of size $\nicefrac{1}{2}$ attached at the base $v_{a}\in\D_{a}$.
Decorated $\Z$-graphs are also studied in \cite{BanSof_GLT2026,Bandb},
where we develop a suitable gap labelling theory for them and solve
the corresponding dry ten Martini problem.
\end{defn}

\begin{example}[Sturmian comb]
\label{exa: Sturm-comb} A particular family of decorated $\Z$-graphs
is obtained by taking the Sturmian subshift $\Omega_{\alpha}$ over
the alphabet $\left\{ 0,1\right\} $ from Example \ref{exa:Sturmian},
and choosing the first decoration to be a dangling edge of length
$\ell$, and the second decoration to be a single vertex (see bottom
of Figure \ref{fig: TilingGraphs}). We call such graphs Sturmian
combs and demonstrate in Section \ref{sec:Discussion-and-Example}
how our results apply for this family of graphs.
\end{example}

The family of graphs $\Gamma_{\Omega}:=\left(\Gamma_{\omega}\right)_{\omega\in\Omega}$
is equipped with a naturally induced shift, 
\begin{align}
 & \sft:\Gamma_{\Omega}\rightarrow\Gamma_{\Omega},\label{eq:randomgraph1}\\
 & \sft\Gamma_{\omega}=\Gamma_{\sft\omega},\label{eq:randomgraph2}
\end{align}
where we allowed ourselves to abuse the shift notation $\sft$. Equipping
the tiling graphs $\Gamma_{\omega}\in\Gamma_{\Omega}$ with the Kirchhoff
Laplacian $H_{\omega}$, one can consider the family $H_{\Omega}:=\left(H_{\omega}\right)_{\omega\in\Omega}$
as a dynamical system of operators. Denoting by $\sft^{\sharp}:C(\Gamma_{S\omega})\rightarrow C(\Gamma_{\omega})$
the pullback of $\sft$, it is easy to check that the family $H_{\Omega}$
is covariant, i.e.,
\begin{equation}
H_{\sft\omega}=\left(\sft^{\sharp}\right)^{-1}H_{\omega}\sft^{\sharp},\,\forall\omega\in\Omega,\label{eq:covariant}
\end{equation}
which means that the operators $H_{\omega}$ and $H_{\sft\omega}$
are unitarily equivalent. As we shall see in Theorem \ref{thm:TMP},
minimality of $\Omega$ implies that $\sigma\left(H_{\omega}\right)$
is in fact independent of $\omega\in\Omega$, and so we can simply
denote it by $\sigma\left(H_{\Omega}\right)$.

Throughout the paper, we make the following technical assumption about
the tiles:
\begin{assumption}
\label{assu:no-deg-2}There is at most one tile in $\left(\Ta\right)_{a\in\mathcal{A}}$
in which $v^{1}_{a}$ and $v^{2}_{a}$ are connected by a single edge.
In particular, at most one tile is an interval. Furthermore, no tile
has degree-$2$ vertices, except for possibly the boundary vertices
$v^{a}_{1}$ and $v^{a}_{2}$.

The following assumptions are about the tile edge lengths:
\begin{enumerate}
\item Assume that for each $a\in\A$, the quantity
\begin{equation}
\ell^{a}_{\min}:=\min\set{\ell_{e}}{e\sim v^{1}_{a}}\label{eq:min-l}
\end{equation}
is realized by a unique edge, and that $\ell^{a}_{\min}\neq\ell^{a'}_{\min}$
for all $a\neq a'$.
\item Denote 
\[
\ell^{a}_{\textrm{left}}:=\begin{cases}
\ell^{a}_{\min} & \deg(v^{1}_{a})=1\\
0 & \textrm{else}
\end{cases}.
\]
If one of the tiles is an interval, denote it by $T_{\overline{a}}$
and its length by $\ell^{\overline{a}}$. Assume that for all $a',a''\in\A\backslash\{\overline{a}\}$,
the equality $\ell^{a'}_{\textrm{left}}\equiv\ell^{a''}_{\textrm{left}}\mod{\ell^{\overline{a}}}$
implies that $\ell^{a'}_{\textrm{left}}=\ell^{a''}_{\textrm{left}}$.
\\
If no tile is an interval, but there is a tile in which the boundary
vertices are connected by a single edge, denote it by $T_{\overline{a}}$
and the edge length connecting $v^{1}_{\a}$ and $v^{2}_{\a}$ by
$\ell^{\overline{a}}$. Assume that for all $a',a''\in\A\backslash\{\overline{a}\}$,
$\ell^{a'}_{\textrm{left}}-\ell^{a''}_{\textrm{left}}\neq\ell^{\overline{a}}$.
\end{enumerate}
\end{assumption}

\begin{rem*}
The assumption of excluding vertices of degree $2$ is a common one
when considering the Kirchhoff Laplacian. Indeed, a degree-$2$ vertex
with Kirchhoff conditions may be removed by replacing its two incident
edges with a single edge whose length is the sum of their lengths.
This does not affect the Laplacian, up to unitary equivalence. All
the assumptions on the edge lengths included in Assumption~\ref{assu:no-deg-2}
may be replaced by the simple (yet stronger) assumption that there
are no nontrivial integer linear relations among the tile edge lengths.
We prefer not to use this stronger assumption, as it excludes some
graph families, such as the decorated graphs, from the forthcoming
results.
\end{rem*}
\begin{thm}
\label{thm:TMP}Let $\Omega$ be a minimal aperiodic subshift satisfying
Boshernitzan's condition. Let $\left(\Gamma_{\omega}\right)_{\omega\in\Omega}$
be an associated family of metric tiling graphs, with tiles satisfying
Assumption \ref{assu:no-deg-2}. Then:

1. There exists a closed set of zero Lebesgue measure $\sigma\left(H_{\Omega}\right)\subset\R$
such that $\sigma\left(H_{\omega}\right)=\sigma\left(H_{\Omega}\right)$
for all $\omega\in\Omega$. In particular, the absolutely continuous
spectrum $\sigma_{ac}\left(H_{\omega}\right)$ is empty for all $\omega\in\Omega$.

2. There exists a (possibly empty) set of isolated eigenvalues $F\subset\R$
such that $\sigma\left(H_{\Omega}\right)\backslash F$ is a generalized
Cantor set.
\end{thm}

\begin{rem}
The fact that $\sigma\left(H_{\omega}\right)$ is independent of $\omega$
is standard for minimal dynamical systems, and does not require Assumption~\ref{assu:no-deg-2}
or Boshernitzan's condition.
\end{rem}

By \textit{a generalized} Cantor set, we mean a closed, nowhere-dense
subset of $\R$ with no isolated points (relaxing the usual condition
of compactness). A precise classification of the set $F$ will be
given in Section \ref{sec:Cantor spectrum proof} (see Remark \ref{rem:thm1.7.4.5}).

Consider the set of all edge lengths $\mathbb{R}^{n}_{+}$ and recall
that a subset $A\subset\mathbb{R}^{n}_{+}$ is called Baire-generic
if it contains a countable intersection of dense open subsets. Noting
that the additional condition on the tile edge lengths in Assumption
\ref{assu:no-deg-2} is satisfied for a Baire-generic and Lebesgue
almost-sure choice of the lengths, we immediately deduce the following:
\begin{cor}
\label{cor: generic_tiling_graphs}For a Baire-generic and Lebesgue
almost-sure choice of the tile edge lengths, the resulting tiling
graph satisfies that $\sigma\left(H_{\Omega}\right)$ has Lebesgue
measure zero and, apart from a discrete set of isolated eigenvalues,
forms a generalized Cantor set.
\end{cor}

We further show that when restricting to the important class of decorated
$\Z$-graphs (see Definition \ref{def: decorated lines}), the last
statement can be strengthened.
\begin{thm}
\label{thm:Cantor}Let $\left(\Omega,\sft\right)$ be a minimal aperiodic
subshift satisfying Boshernitzan's condition, and let $\left(\Gamma_{\omega}\right)_{\omega\in\Omega}$
be an associated family of decorated $\Z$-graphs. Assume that the
decorations do not contain loop edges (i.e., an edge from a vertex
to itself). Then, for a Baire-generic choice of the decoration edge
lengths, $\sigma\left(H_{\Omega}\right)\subset\R$ is a generalized
Cantor set of zero Lebesgue measure.
\end{thm}

As we shall see in Subsection~\ref{subsec:Generic-Cantor-spectrum},
$\sigma\left(H_{\omega}\right)$ may not be a Cantor set if $H_{\omega}$
admits infinitely degenerate eigenvalues with compactly supported
eigenfunctions, a case that may happen for particular (non-generic)
choices of the edge lengths, or in the presence of loop edges.

We mention that the questions we consider here may also be considered
for more general Schrödinger operators than the Laplacian, namely
of the form $H_{\omega}=-\frac{d^{2}}{dx^{2}}+V_{\omega}\left(x\right)$
where $V_{\omega}\left(x\right)$ is sufficiently ``compatible''
with the tiling structure of $\Gamma_{\omega}$, similar to the continuum
Schrödinger operators studied in \cite{DamFilGor_ahp15}. We do not
pursue this extension here.

\subsection*{Acknowledgments}

We would like to thank Lior Alon, Siegfried Beckus, Gregory Berkolaiko,
David Damanik, Jake Fillman, Pavel Kurasov, and Selim Sukhtaiev for
extremely helpful discussions. We gratefully acknowledge the hospitality
of the Institute of Mathematics at the University of Potsdam, where
some of this research took place. The research for this paper was
partially conducted at the Israel Institute for Advanced Studies,
as part of the Research Group Analysis, Geometry, and Spectral Theory
of Graphs during 2025-2026. This research was supported by the Israel
Science Foundation (ISF Grant No. 2362/25), and by grant no. 2024154
from the United States - Israel Binational Science Foundation (BSF),
Jerusalem, Israel. GS was supported by the STSM travel grant as part
of the COST Action CA18232.

\section{The $m$-function and a Borg-Marchenko result\label{sec:BM theorem}}

In this section, we introduce a main tool for the proof of Theorem
\ref{thm:TMP}, which is the $m$-function associated with a quantum
graph (also known as the Weyl-Titchmarsh function, or Dirichlet-to-Neumann
map). We then prove that the $m$-function of a tiling graph $\Gamma_{\omega}$
can be used to recover the sequence of tiles $\omega\left(n\right)$.
Such an inverse spectral geometric result is commonly known as a Borg-Marchenko
type result, see \cite{AvdKur_ipi08,Kurasov_inverse_2009,Kurasov_Book}
for some relevant results in the context of quantum graphs, or \cite{marchenko1950certain}
for the original paper by Marchenko.

\subsection{The $m$-function and the Borg-Marchenko result\label{subsec:m-function}}

We now introduce the $m$-functions associated with $\Gamma_{\omega}$.
For this, we start by partitioning the infinite graph $\Gamma_{\omega}$
into two half-infinite graphs $\Gamma^{\pm}_{\omega}$, which are
obtained by cutting $\Gamma_{\omega}$ at the origin $o\left(\Gamma_{\omega}\right)$.

Let $z\in\C_{+}:=\left\{ z\in\C:\mathrm{Im}\left(z\right)>0\right\} $.
On each half-graph $\Gamma^{\pm}_{\omega}$, consider the differential
equation
\begin{equation}
-\frac{d^{2}}{dx^{2}}f_{\pm}=zf_{\pm},\label{eq:m-ODE}
\end{equation}
with the Kirchhoff condition imposed at all vertices except for $o\left(\Gamma^{\pm}_{\omega}\right)$,
where no boundary condition is imposed. Since $z\notin\R$, this equation
has a unique solution $f^{\pm}_{\omega}\left(\thinspace\cdot\thinspace;z\right)\in L^{2}\left(\Gamma^{\pm}_{\omega}\right)$
(the proof is similar to that of \cite[lem. 9.7, thm. 9.1]{Teschl2014}).
The $m^{\pm}_{\omega}$ functions are then defined as follows:
\begin{align}
 & m^{\pm}_{\omega}:\C_{+}\rightarrow\C,\label{eq:m2}\\
 & m^{\pm}_{\omega}\left(z\right)=\frac{\sum_{e\sim o}\left.\frac{df^{\pm}_{\omega}}{dx}\right|_{e}\left(o;z\right)}{f^{\pm}_{\omega}\left(o;z\right)},\label{eq:m3}
\end{align}
where we abbreviate notation by using $o$ instead of $o\left(\Gamma^{\pm}_{\omega}\right)$,
the edges $e\sim o$ are the ones incident to $o$, and the derivatives
are taken in the outward pointing direction from $o$ (the $\pm$
sign affects the chosen edges and the derivative directions). By standard
theory (see \cite[thm. 18.2]{Kurasov_Book}), $m^{\pm}_{\omega}\left(z\right)$
are Herglotz functions, i.e., they are analytic and their image is
contained in $\C_{+}$.

Our goal in this section is to prove the following Theorem:
\begin{thm}[Borg-Marchenko]
\label{thm:BM}Assume that the individual tiles $\left(\Ta\right)_{a\in\mathcal{A}}$
are given and satisfy the conditions of Theorem \ref{thm:TMP}. Then
for any $\omega\in\A^{\Z}$, the pair of corresponding $m$-functions
$m^{\pm}_{\omega}\left(z\right)$ uniquely determines $\omega$.
\end{thm}

This result essentially states that one can reconstruct the graphs
$\Gamma^{\pm}_{\omega}$ from the functions $m^{\pm}_{\omega}\left(z\right)$,
assuming that the basic building blocks (tiles) are known. The proof
relies on a periodic orbit expansion for $m^{\pm}_{\omega}\left(z\right)$,
which is developed next. Starting from here and in the following subsections,
we conveniently focus on the half-graph $\Gamma^{+}_{\omega}$ (and
the associated $m$-function $m^{+}_{\omega}\left(z\right)$), noting
that $\Gamma^{-}_{\omega}$ and $m^{-}_{\omega}\left(z\right)$ may
be treated similarly. We also keep $\omega$ fixed, and thus from
now on and until the end of the section, we abbreviate and simply
write $o$ instead of $o\left(\Gamma^{+}_{\omega}\right)$.

In order to state the periodic orbit expansion for $m^{+}_{\omega}\left(z\right)$,
we first introduce a definition of a \textit{tailed} periodic orbit
(see also Figure~\ref{fig: periodic orbit}).
\begin{defn}
\label{def:periodic orbits} ~
\begin{enumerate}
\item Let $\Gamma$ be a metric graph. A \textit{tailed periodic orbit}
(TPO) $p$ around a vertex $v\in\Gamma$ is a finite sequence of successively
connected directed edges in $\Gamma$, $\left(v_{1},v_{2}\right)$,$\left(v_{2},v_{3}\right),...,\left(v_{n},v_{n+1}\right)$,
where $v_{1}=v_{n}=v$ (see Figure~\ref{fig: periodic orbit}). We
think of such TPOs as periodic orbits around $v$, along with a chosen
terminal direction $\left(v_{n},v_{n+1}\right)$ (thought of as the
``tail'' of the orbit). In this work, we mostly consider TPOs around
the origin of $\Gamma^{\pm}_{\omega}$ (so $v=o$).
\item For a TPO $p$, we denote its total length by $\ell_{p}:=\sum^{n-1}_{i=1}\left|\left(v_{i},v_{i+1}\right)\right|$,
where $\left|\cdot\right|$ denotes the edge length (Note that we
do not include the tail $\left(v_{n},v_{n+1}\right)$ in the sum).
We denote its topological length by $\left|p\right|:=n-1$, i.e.,
the number of directed edges in the sequence $p$, excluding the tail
$\left(v_{n},v_{n+1}\right)$.
\item Let $\Gamma^{+}_{\omega}$ be a half-infinite tiling graph, with origin
vertex $o$. The set of TPOs around the origin for $\Gamma^{+}_{\omega}$
is denoted by $\mathcal{P}^{+}_{\omega}$.
\end{enumerate}
\end{defn}

\begin{figure}
\includegraphics[scale=0.55]{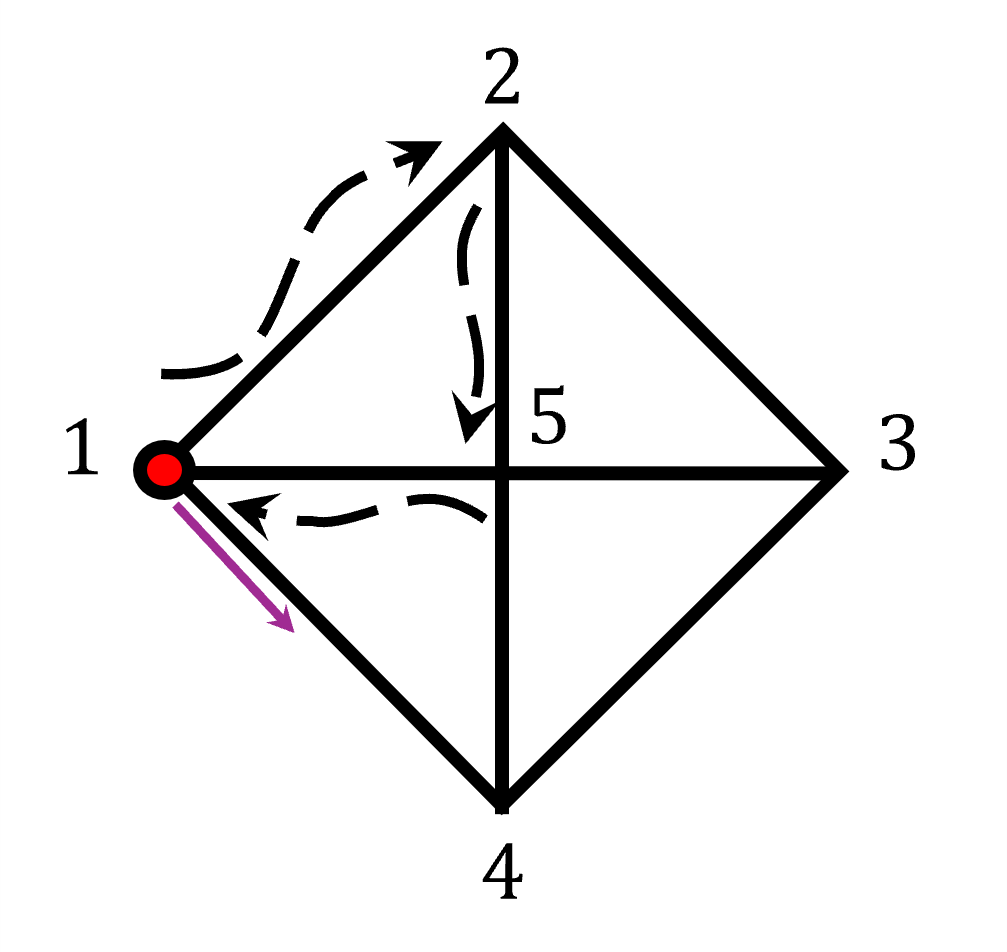}

\caption{A square kite graph with vertex set $\left\{ 1,2,3,4,5\right\} $,
together with a TPO $p=\left(\left(1,2\right),\left(2,5\right),\left(5,1\right),\left(1,4\right)\right)$
around the marked vertex $1$ (the terminal direction is marked in
purple). \label{fig: periodic orbit}}
\end{figure}

\begin{thm}
\label{thm:periodic-orbits} The $m$-function $m^{+}_{\omega}\left(z\right)$
has the following expansion in terms of the tailed periodic orbits
of $\Gamma^{+}_{\omega}$:
\begin{equation}
m^{+}_{\omega}\left(z\right)=-\sqrt{-z}\sum_{p\in\mathcal{P}^{+}_{\omega}}A_{p}e^{-\sqrt{-z}\ell_{p}}.\label{eq:periodic-m}
\end{equation}
\end{thm}

The coefficients $A_{p}$ appearing in the theorem are explicitly
given in the proof. Theorem \ref{thm:periodic-orbits} is fully proven
in the following subsection. A reader who is less interested in the
proof details may proceed directly to Subsection~\ref{subsec:bm-proof},
where the periodic orbit expansion (\ref{eq:periodic-m}) is applied
in the proof of the Borg-Marchenko Theorem \ref{thm:BM} to recover
the sequence $\omega$.

We remark that in many cases (see e.g. \cite[thm 9.22]{Teschl2014}
and \cite[thm 3.5]{DamFanSuk_ahp20}), Borg-Marchenko results on the
line rely on first proving asymptotic estimates for $m\left(z\right)$
as the spectral parameter $z$ tends to infinity. Theorem \ref{thm:periodic-orbits}
not only proves the Borg-Marchenko Theorem \ref{thm:BM} directly
(with no need for such asymptotics), but can also be used to provide
such asymptotics for $m^{\pm}\left(z\right)$ as $z\rightarrow\infty$,
as in the alternative proofs.

\subsection{The dynamical Dirichlet-to-Neumann map and the proof of the TPO expansion}

We start by introducing the relevant definitions related to the dynamical
Dirichlet-to-Neumann map (also known as the dynamical response operator),
and refer to \cite[sec. 19]{Kurasov_Book} for a more thorough background.
We then use it to prove the TPO expansion for the $m$-function stated
in Theorem \ref{thm:periodic-orbits}.

Consider the dynamical wave equation on $\Gamma^{+}_{\omega}$:
\begin{align}
 & \frac{\partial^{2}u}{\partial x^{2}}=\frac{\partial^{2}u}{\partial t^{2}},\label{eq:wave equation}\\
 & u\left(x,0\right)=0,\label{eq:wave2}\\
 & u_{t}\left(x,0\right)=0,\label{eq:wave3}
\end{align}
with the Kirchhoff condition imposed at all vertices except for $o$,
which is subject to the boundary control:
\begin{equation}
u\left(o,t\right)=b\left(t\right),\label{eq:wave4}
\end{equation}
with $b\in L^{2}_{\text{loc}}\left(0,\infty\right)$. The \textit{dynamical
Dirichlet-to-Neumann (DDTN) map} $\Lambda^{+}_{\omega}$ is the mapping:
\begin{equation}
b\left(\cdot\right)\xmapsto{\Lambda^{+}_{\omega}}\partial_{n}u\left(o,\cdot\right):=\sum_{e\sim o}\frac{\left.\partial u\right|_{e}\left(o,\cdot\right)}{\partial x},\label{eq:ddtn}
\end{equation}
where the edges $e\sim o$ are the ones incident to $o$ in $\Gamma^{+}_{\omega}$,
and the derivatives are taken in the outward pointing direction from
$o$. By \cite[eq. (19.12)]{Kurasov_Book}, the DDTN map and the $m$-function
are in one-to-one correspondence, given by:
\begin{equation}
m^{+}_{\omega}\left(-s^{2}\right)=\frac{\widehat{\left(\Lambda^{+}_{\omega}b\right)}\left(s\right)}{\widehat{b}\left(s\right)},\label{eq:DDTN-M-correspondence}
\end{equation}
where $\widehat{\left(\cdot\right)}$ is the Laplace transform:
\begin{equation}
\hat{f}\left(s\right)=\int^{\infty}_{0}f\left(t\right)e^{-st}\rmd t.\label{eq:Laplace}
\end{equation}

Our next goal is to obtain an expansion for $\Lambda^{+}_{\omega}$
(and hence for $m^{+}_{\omega}$) in terms of tailed periodic orbits
(TPOs)\@. Before doing so, we present the scattering matrix of a
quantum graph (which was first introduced in \cite{KotSmi_ap99,KotSmi_prl97}).
Let $\Gamma$ be a quantum graph and let $v\in\V\left(\Gamma\right)$.
On each edge $e$ incident to $v$, one can write the solutions to
the eigenvalue equation $-\frac{d^{2}}{dx^{2}}g=k^{2}g$ in the form
\begin{equation}
g|_{e}\left(x\right)=a_{e}\rme^{\rmi kx}+a_{\hat{e}}\rme^{\rmi k\left(\ell_{e}-x\right)},\label{eq:scatterwave}
\end{equation}
with $a_{e},a_{\hat{e}}\in\C$. Here, we think of $e$ as a directed
edge pointing towards a vertex $v$, and denote its reversal by $\hat{e}$.
Note that for $g$ to also satisfy the Kirchhoff vertex condition
at $v$, the vectors $\vec{a}_{\text{in}}:=\left(a_{e_{1}},...,a_{e_{\text{deg\ensuremath{\left(v\right)}}}}\right)$,
$\vec{a}_{\text{out}}:=\left(a_{\hat{e}_{1}},...,a_{\hat{e}_{\text{deg\ensuremath{\left(v\right)}}}}\right)$
must satisfy a particular linear relation. The \emph{scattering matrix}
associated with $v$ is the unique unitary matrix $S_{v}\left(k\right)\in U\left(\deg\left(v\right)\right)$
that expresses this linear relation:
\begin{equation}
\vec{a}_{\text{out}}=S_{v}\left(k\right)\rme^{\rmi k\vec{L}}\vec{a}_{\text{in}},\label{eq:a=00003DUa}
\end{equation}
where $\vec{L}=diag\left(\ell_{e_{1}},...,\ell_{e_{\deg\left(v\right)}}\right)$
is the diagonal matrix of edge lengths.

The scattering matrix generally depends on the imposed vertex conditions.
Given two directed edges $e,e'$, we write that $e\rightarrow e'$
at $v$ if the vertex $v$ is both the end vertex of $e$ and the
starting vertex of $e'$. One can show that for the Kirchhoff condition,
$S_{v}\left(k\right)=:S_{v}$ is independent of $k$ and is given
by (cf. \cite[eq. (3.6)]{GnuSmi_ap06}),
\begin{equation}
\left[S_{v}\right]_{e',e}=\begin{cases}
\frac{2}{\deg\left(v\right)}-1, & e'=\hat{e},\\
\frac{2}{\deg\left(v\right)}, & e\rightarrow e'\text{ at \ensuremath{v} and }e'\neq\hat{e}\\
0, & \text{Otherwise.}
\end{cases},\label{eq:kirch-scar}
\end{equation}

The scattering matrix is a common tool in the theory of quantum graphs,
not only for spectral analysis, but also for describing wave propagation,
which we indeed employ in the following proof.
\begin{proof}[Proof of Theorem \ref{thm:periodic-orbits}]
 We show that the DDTN map can be represented as a convolution with
a generalized kernel:
\begin{equation}
\Lambda^{+}_{\omega}b\left(t\right)=\left(k^{+}_{\omega}*b\right)\left(t\right):=\int^{\infty}_{0}k^{+}_{\omega}\left(t-s\right)b\left(s\right)\rmd s.\label{eq:ddtnker}
\end{equation}
Combining this with (\ref{eq:DDTN-M-correspondence}) and using that
the Laplace transform is multiplicative with respect to convolution
gives the simple relation between the kernel $k^{+}_{\omega}\left(t\right)$
and the $m$-function:
\begin{equation}
m^{+}_{\omega}\left(-s^{2}\right)=\frac{\widehat{\left(k^{+}_{\omega}*b\right)}\left(s\right)}{\widehat{b}\left(s\right)}=\widehat{k^{+}_{\omega}}\left(s\right).\label{eq:kernel-M-correspondence}
\end{equation}
Hence, to prove the theorem we show (\ref{eq:ddtnker}) and develop
a TPO expansion for the generalized kernel, $k^{+}_{\omega}\left(t\right)$.
We do both simultaneously.

The key idea is that the solution $u\left(x,t\right)$ to the dynamical
wave equation (\ref{eq:wave equation}) consists of waves propagating
along the edges of the graph, with different amplitudes corresponding
to different TPOs. We systematically construct the expansion of $k^{+}_{\omega}\left(s\right)$
by analyzing the solution to the dynamical wave equation on $\Gamma^{+}_{\omega}$,
keeping track of how it scatters at each vertex. Due to the finite
propagation speed of the wave equation (which equals to $1$ by (\ref{eq:wave equation})),
the solution at time $t$ is always supported on the geodesic ball
of radius $t$ around the origin, which allows us to focus on a bounded
neighborhood of the origin.

~

\textbf{\textit{Step 1: Short-time behavior of the wave equation.}}

Recall that the d'Alembert solution to the dynamical wave equation
is given as a superposition of an ingoing and outgoing wave on each
edge:
\begin{equation}
u|_{e}\left(x,t\right)=F_{e}\left(t-x\right)+G_{e}\left(t+x\right).\label{eq:Dalambert}
\end{equation}
 Through direct substitution into (\ref{eq:wave equation}),(\ref{eq:wave2}),(\ref{eq:wave3}),(\ref{eq:wave4}),
one can verify that for $0<t\leq\min_{e\sim o}\left\{ \left|e\right|\right\} $,
the solution to the wave equation is given by
\begin{equation}
u|_{e}\left(x,t\right)=\begin{cases}
b\left(t-x\right), & e\sim o,\\
0, & \text{otherwise}.
\end{cases}\label{eq:1wave}
\end{equation}
In other words, at short times, the wave only propagates along edges
that are incident to the origin $o$, and the associated d'Alembert
solution is then given by an outgoing wave with profile equal to that
of the boundary control $b\left(t\right)$. This means that at short
times, the DDTN satisfies: 
\begin{align}
 & \Lambda^{+}_{\omega}b\left(t\right)=\partial_{n}u\left(o,t\right)=-\sum_{e\sim o}b'\left(t\right),\label{eq:2wave}
\end{align}
and so we may choose the representing kernel to be
\begin{equation}
k^{+}_{\omega}\left(t\right)=-\deg\left(o\right)\cdot\delta'\left(t\right)+...,\label{eq:3wave}
\end{equation}
where the additional terms vanish for $t<\min_{e\sim o}\left\{ \left|e\right|\right\} $,
and the equality above is in the sense of distributions. Here, $\delta'$
is a delta prime distribution, defined through
\begin{equation}
\left(\delta'*f\right):=f'\left(0\right).\label{eq:delta-prime}
\end{equation}

~

\textbf{\textit{Step 2: First scattering event.}}

Let $v$ be the vertex which is the nearest to the origin and denote
$L:=\left|(o,v)\right|$. Then, once the wave $u|_{e}\left(x,t\right)$
reaches $v$ at time $t=L$, it scatters to all neighboring edges
$e'\sim v$ (see Figure \ref{fig: scattering}), where the associated
wave amplitudes at each edge are determined by the Kirchhoff condition
at $v$. It is simple to verify using Equations (\ref{eq:scatterwave}),(\ref{eq:a=00003DUa})
that for a plane wave solution to the wave equation of the form
\begin{equation}
w|_{e}\left(x,t\right)=\rme^{\rmi\left(t-kx\right)},\label{eq:plane-wave}
\end{equation}
the wave amplitude at $e$ following this event is determined via
the scattering matrix element $\left[S_{v}\right]_{\hat{e},e}$, i.e.,
\begin{equation}
w_{e}\left(x,t\right)=\rme^{{\rm \rmi}\left(t-kx\right)}+\left[S_{v}\right]_{\hat{e},e}\rme^{\rmi\left(\left(t-L\right)-\left(L-kx\right)\right)}.\label{eq:plane-scatter}
\end{equation}
Then, straightforward application of Fourier theory (writing $u|_{e}\left(x,t\right)$
as an infinite superposition of plane waves) shows that similarly,
our d'Alembert solution at the edge $e$ after the first scattering
event is given by
\begin{equation}
u|_{e}\left(x,t\right)=b\left(t-x\right)+\left[S_{v}\right]_{\hat{e},e}b\left(\left(t-L\right)-\left(L-x\right)\right),\label{eq:4wave}
\end{equation}
see also \cite[sec. 19.5]{Kurasov_Book} for additional details. Mainly,
the scattering process at this vertex introduces a secondary wave
that backscatters into its initial edge $e$ with amplitude $\left[S_{v}\right]_{\hat{e},e}$,
but with a reversed direction (see Figure \ref{fig: scattering}).

\begin{figure}
\includegraphics[scale=0.55]{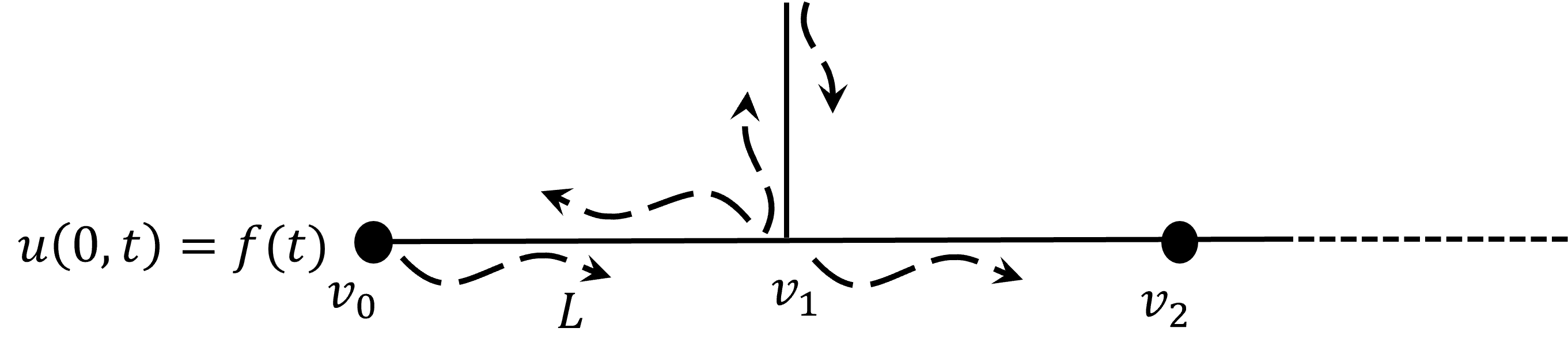}

\caption{Wave scattering on a comb graph at short times. Initially, the solution
propagates along the directed edge $e=\left(v_{0},v_{1}\right)$,
until it reaches the vertex $v_{1}$ at time $t=L$. The solution
then scatters in three directions, and in particular backscatters
into the directed edge $\left(v_{1},v_{0}\right)$ with wave amplitude
$\left[S_{v_{1}}\right]_{\hat{e},e}=\frac{2}{\deg\left(v_{1}\right)}-1=-\frac{1}{3}$.
\label{fig: scattering}}
\end{figure}
At time $t=2L$, the secondary wave reaches the origin $o$, and similarly
to before, scatters in all directions $e'\sim o$ with amplitude $\left[S_{o}\right]_{e',\hat{e}}$
due to the Kirchhoff condition at the origin. At short times after
this event\footnote{i.e. times with respect to which there is no additional scattering
event which makes another wave reach the origin.}, the normal derivative of $u\left(x,t\right)$ at the origin is given
by:
\begin{align}
 & \partial_{n}u\left(o,t\right)=-\sum_{e'\sim o}\left(b'\left(t\right)+\left[S_{o}\right]_{e',\hat{e}}\left[S_{v}\right]_{\hat{e},e}b'\left(t-2L\right)\right).\label{eq:5wave}
\end{align}
Thus, for $t\in\left[0,2L\right]$, we may choose the kernel to be
\begin{equation}
k^{+}_{\omega}\left(t\right)=-\deg\left(o\right)\delta'\left(t\right)-\left(\sum_{e'\sim o}\left[S_{o}\right]_{e',\hat{e}}\left[S_{v}\right]_{\hat{e},e}\right)\delta'\left(t-2L\right).\label{eq:6wave}
\end{equation}
Note that the terms $\delta'\left(t\right)$ and $\delta'\left(t-2L\right)$
correspond exactly to the shortest TPOs around the origin: $\delta'\left(t\right)$
corresponds to the trivial TPO of length $0$ and $\delta'\left(t-2L\right)$
correspond to all TPOs consisting of going back and forth along $e$
and an arbitrary tail, $e'\sim o$.\\
~

\textbf{\textit{Step 3: Scattering in longer times.}}

One can continue expanding the solution $u\left(x,t\right)$ in the
same fashion. As time progresses beyond $t=2L$, the d'Alembert solution
$u\left(x,t\right)$ continues to propagate and undergoes additional
scattering events at vertices further away from the origin. On edges
adjacent to $o$, $u\left(x,t\right)$ will always consist of a superposition
of delayed waves that scattered across some TPO around the origin.
The term corresponding to a TPO $p$ will be of the form $A_{p}b\left(t-\ell_{p}\right)$,
where the delay time $\ell_{p}$ is equal to the metric length of
the TPO $p$, and the amplitude $A_{p}$ is the product of scattering
coefficients along the TPO:
\begin{equation}
A_{p}:=\prod^{\left|p\right|}_{j=1}\left[S_{v_{j+1}}\right]_{\left(v_{j+1},v_{j+2}\right),\left(v_{j},v_{j+1}\right)}.\label{eq:scat-amp-2}
\end{equation}
On the level of the DDTN kernel $k^{+}_{\omega}\left(s\right)$, this
gives an infinite TPO expansion:
\begin{align}
 & \partial_{n}u\left(o,t\right)=-\sum_{p\in\mathcal{P}^{+}_{\omega}}A_{p}b'\left(t-\ell_{p}\right),\label{eq:2ddtn}\\
\Rightarrow\quad & k^{+}_{\omega}\left(t\right)=-\sum_{p\in\mathcal{P}^{+}_{\omega}}A_{p}\delta'\left(t-\ell_{p}\right).\label{eq:ddtn-ker-2}
\end{align}

Equation (\ref{eq:periodic-m}) can now be obtained by substituting
(\ref{eq:ddtn-ker-2}) into (\ref{eq:kernel-M-correspondence}) and
using the identity $\widehat{\delta'\left(t-\ell\right)}\left(s\right)=se^{-\ell s}$,
which finally gives:
\begin{equation}
m^{+}_{\omega}\left(z\right)=-\sqrt{-z}\sum_{p\in\mathcal{P}^{+}_{\omega}}A_{p}e^{-\sqrt{-z}\ell_{p}},\label{eq:m-final}
\end{equation}
finishing the proof.
\end{proof}

\subsection{Proving the Borg-Marchenko result\label{subsec:bm-proof}}
\begin{proof}[Proof of the Theorem \ref{thm:BM} (Borg-Marchenko)]
 Let $\left(\omega\left(n\right)\right)^{\infty}_{n=0}\in\A^{\N}$.
We use the TPO expansion (\ref{eq:periodic-m}) of $m^{+}_{\omega}\left(z\right)$
in order to inductively recover each element of $\left(\omega\left(n\right)\right)^{\infty}_{n=0}$.
Hence, the graph which is considered below is $\Gamma^{+}_{\omega}$.
In exactly the same manner $m^{-}_{\omega}\left(z\right)$ may be
used to recover $\left(\omega\left(n\right)\right)^{-\infty}_{n=-1}$.

\uline{Base case, \mbox{$n=0$}:}

By Assumption \ref{assu:no-deg-2}, there is at most one $\bar{a}\in\A$
in which $v^{1}_{\bar{a}}$ and $v^{2}_{\bar{a}}$ are connected by
a single edge (such a tile might be an interval). Consider the TPO
expansion of $m^{+}_{\omega}\left(z\right)$. First, note that if
$\omega\left(0\right)=a\neq\a$, then setting
\begin{equation}
\ell^{a}_{\min}:=\min\left\{ \ell_{e}:e\sim v^{1}_{a}\right\} ,\label{eq:lmin}
\end{equation}
the term $-A_{a}\sqrt{-z}e^{-\sqrt{-z}2\ell^{a}_{\min}}$ appears
in this TPO expansion with $A_{a}\neq0$. Indeed, in this case, there
are exactly $\deg\left(o\right)$ TPOs of length $2\ell^{a}_{\min}$.
These are the TPOs that traverse the (unique) edge $e_{\min}\sim o$
of length $\ell^{a}_{\min}$, backscatter and return to the origin,
and then scatter to some other edge $e'\sim o$. Since we assume that
$a\neq\a$, we get that there exists a vertex $v_{\min}$ in $\Ta$
such that $e_{\min}=\left\{ o,v_{\min}\right\} $, then direct computation
using (\ref{eq:kirch-scar}) and the second term in (\ref{eq:6wave})
shows that the total scattering amplitude contribution of these shortest
TPOs is given by 
\begin{align}
 & C_{a}=\sum_{e'\sim o}\left[S_{o}\right]_{e',\hat{e}}\left[S_{v_{\min}}\right]_{\hat{e},e}\label{eq:scattering-prefactor}\\
 & =\left(\frac{2}{\deg\left(v_{\min}\right)}-1\right)\left(\frac{2}{\deg\left(o\right)}-1\right)+\sum_{e_{\min}\neq e'\sim o}\left(\frac{2}{\deg\left(v_{\min}\right)}-1\right)\cdot\frac{2}{\deg\left(o\right)}\nonumber \\
 & =\left(\frac{2}{\deg\left(v_{\min}\right)}-1\right)\left(\frac{2}{\deg\left(o\right)}-1+\frac{2(\deg\left(o\right)-1)}{\deg\left(o\right)}\right)\nonumber \\
 & =\frac{2}{\deg\left(v_{\min}\right)}-1\neq0,\nonumber 
\end{align}
where the last inequality follows since $v_{\min}$ is included in
$\Ta$ (by $a\neq\a$) and Assumption~\ref{assu:no-deg-2} asserts
that $\deg\left(v_{\min}\right)\geq3$.

Now, iterate over the list $\left(\ell^{a}_{\min}\right)_{a\in\A\backslash\{\bar{a}\}}$
in increasing order, keeping in mind that no two elements in this
list are equal, by Assumption \ref{assu:no-deg-2}. For each $\ell^{a}_{\min}$,
examine the corresponding term $-A_{a}\sqrt{-z}e^{-\sqrt{-z}2\ell^{a}_{\min}}$
in the TPO expansion (\ref{eq:periodic-m}) of $m^{+}_{\omega}\left(z\right)$
(with the prefactor $A_{a}$ being possibly zero). By the discussion
above, for each $a\in\A\backslash\{\bar{a}\}$ such that $A_{a}=0$,
we can immediately deduce that $\omega\left(0\right)\neq a$. Consider
the minimal $\ell^{a}_{\min}$ with $a\in\A\backslash\{\bar{a}\}$
such that $A_{a}\neq0$. Then there exists a TPO of length $2\ell^{a}_{\min}$
around $o$, and hence the unique edge of length $\ell^{a}_{\min}$
must lie in the first tile, implying that $\omega\left(0\right)=a$.
If, on the other hand, $A_{a}=0$ for all $a\in\A\backslash\{\a\}$,
then the only remaining possibility is $\omega\left(0\right)=\a$.
We note that by Assumption~\ref{assu:no-deg-2} it is not possible
that there is an arrangement of a few tiles (more than one) which
contains a TPO of just two scattering events, such that this TPO's
length equals $2\ell^{a}_{\min}$ for some $a\in\A$. The impossibility
of such arrangements validates that the process described above indeed
yields the correct $\omega(0)$.\uline{}\\
\uline{Induction step:}

Suppose that $\omega\left(0\right),...,\omega\left(n-1\right)$ are
known, and we wish to recover $\omega\left(n\right)$. Let $p$ be
a TPO of minimal length that travels from the origin (which is also
$v^{1}_{\omega\left(0\right)}$) to the right boundary vertex $v^{2}_{\omega\left(n-1\right)}$,
and then returns to the origin along the reversal of the same path
(there might be several such TPOs, either differing only by their
tail, or by other edges). We now write an expansion of all TPOs of
length $\ell_{p}$ which are contained in the first known $n$ tiles.
In this expansion there is an unknown parameter which is the degree
of the right boundary vertex of the TPO. The degree of that vertex
within $\Gamma^{+}_{\omega}$ equals $\deg(v^{2}_{\omega\left(n-1\right)})+\deg(v^{1}_{\omega\left(n\right)})$,
but we do not know it until we recover what is $\omega(n)$. Comparing
this TPO expansion to the corresponding term $-A_{p}\sqrt{-z}e^{-\sqrt{-z}\ell_{p}}$
allows to recover the value of $\deg(v^{1}_{\omega\left(n\right)})$.
Doing so, we may already restrict the possible choices for $\omega(n)$
to be in a subset $\A_{0}\subset\A$.

Recall that we take $\bar{a}$ to be such that the corresponding tile
$T_{\a}$ has $v^{1}_{\bar{a}}$ and $v^{2}_{\bar{a}}$ connected
by a single edge; there is at most one such tile (by Assumption~\ref{assu:no-deg-2})
and there might not be such at all. As in the base case, to recover
$\omega\left(n\right)$ we iterate over the list $\left(\ell^{a}_{\min}\right)_{a\in\A_{0}\backslash\{\bar{a}\}}$
in increasing order. For each $a\in\A_{0}\backslash\{\bar{a}\}$,
focus on the term $-A_{a}\sqrt{-z}e^{-\sqrt{-z}\left(\ell_{p}+2\ell^{a}_{\min}\right)}$
in the TPO expansion (\ref{eq:periodic-m}) of $m^{+}_{\omega}\left(z\right)$.
Since we already know the first $n$ tiles, we write an expansion
of all TPOs of length $\ell_{p}+2\ell^{a}_{\min}$ which are contained
in the first known $n$ tiles (if such exist). If these TPOs pass
via $v^{2}_{\omega\left(n-1\right)}$ (and go back into $T_{\omega(n-1)}$
rather than entering $T_{\omega(n)}$) we know the corresponding scattering
amplitude since the value $\deg(v^{2}_{\omega\left(n-1\right)})+\deg(v^{1}_{\omega\left(n\right)})$
is already known, as explained above. Hence, we may explicitly compute
the overall contribution of all TPOs of length $\ell_{p}+2\ell^{a}_{\min}$
contained in the first known $n$ tiles to the TPO expansion of $m^{+}_{\omega}\left(z\right)$.
This yields a term of the form $-\tilde{A_{a}}\sqrt{-z}e^{-\sqrt{-z}\left(\ell_{p}+2\ell^{a}_{\min}\right)}$
(and again, $\tilde{A_{a}}$ may be $0$).

Similar to before, the first $a\in\A_{0}\backslash\{\a\}$ for which
$A_{a}\neq\tilde{A_{a}}$, implies that there are additional TPOs
of length $\ell_{p}+2\ell^{a}_{\min}$ which pass through the $n+1$
tile, $T_{\omega(n)}$. By the choice of $\ell_{p}$ and $\ell^{a}_{\min}$,
such TPOs must consist of the mentioned shortest path from the origin
to $v^{2}_{\omega\left(n-1\right)}$, which go forth and back the
unique edge $e_{\min}$ of length $\ell^{a}_{\min}$ in $T_{\omega(n)}$
adjacent to $v^{1}_{\omega\left(n\right)}$, and then return to the
origin through the initial path (reversed). Since no two elements
in the list $\left(\ell^{a}_{\min}\right)_{a\in\A_{0}\backslash\{\bar{a}\}}$
are equal, we deduce that $\omega\left(n\right)=a$. If instead $A_{a}=\tilde{A_{a}}$
for all $a\in\A_{0}\backslash\{\bar{a}\}$, then $\omega\left(n\right)=\a$.
This completes the reconstruction of the sequence $\omega$.
\end{proof}

\begin{rem}
We emphasize that Theorem \ref{thm:BM} shows that one can reconstruct
$\Gamma_{\omega}$ from $m^{\pm}_{\omega}\left(z\right)$ assuming
that the possible tiles $\left(\Ta\right)_{a\in A}$ are already known.
It is natural to ask when one can recover $\Gamma_{\omega}$ from
$m^{\pm}_{\omega}\left(z\right)$ without apriori knowing the possible
tiles. This is closely tied to recovering the individual (compact)
tiles from their respective $m$-functions. This question has been
addressed in works such as \cite{AvdKur2_ipi10,AvdKur_ipi08,Kurasov_inverse_2009}.
This can, for instance, be done assuming that the tiles $\left(\Ta\right)_{a\in\A}$
are trees, but in general counterexamples exist.
\end{rem}

The following is an immediate corollary of the TPO expansion for $m^{\pm}_{\omega}\left(z\right)$
and the Borg-Marchenko theorem:
\begin{cor}
\label{cor:loc-uni}$\omega_{n}\rightarrow\omega$ in the product
topology on $\A^{\Z}$ if and only if $m^{\pm}_{\omega_{n}}\left(z\right)\rightarrow m^{\pm}_{\omega}\left(z\right)$
locally uniformly. In other words, the mapping $\omega\mapsto m^{\pm}_{\omega}\left(z\right)$
is bijective and continuous.
\end{cor}

This result will be used in the proof of the oracle theorem (Theorem
\ref{thm:oracle}) to pass from identities of $m$-functions to continuity
of the corresponding restriction maps.

The proof of the Borg-Marchenko result above was tailored for tiling
graphs which satisfy the edge length condition of Assumption \ref{assu:no-deg-2}
(which is a generic assumption). There are many other families of
tiles which do not satisfy this assumption. Nevertheless, a similar
Borg-Marchenko result may be proven for those as long as the TPOs
of the different tiles are ``sufficiently distinguishable''. One
class of graphs whose tiles do not satisfy the edge length condition
of Assumption \ref{assu:no-deg-2} is the class of decorated $\Z$-graphs
from Definition \ref{def: decorated lines}. If we adjust the edge
length assumption to this class of tiles, we may get the following
Borg-Marchenko result.
\begin{assumption}
\label{assu:deco-deg}Let $\Gamma_{\omega}$ be a decorated $\Z$-graph.
Suppose that the decorations $\left(\D_{a}\right)_{a\in\A}$ satisfy
that at most one decoration is trivial (i.e., a single vertex), and
that the decorations have no vertices of degree $2$, except for possibly
the base vertex $u_{a}$.
\end{assumption}

\begin{prop}
\label{prop:BM-decorated}Let $\Gamma_{\omega}$ be a decorated $\Z$-graph
as in Assumption \ref{assu:deco-deg}. Suppose that for each $a\in\A$,
the quantity
\begin{equation}
\ell^{a}_{\min}:=\min\left\{ \ell_{e}:e\sim v_{a}\right\} \label{eq:lmin-decorated-line}
\end{equation}
is realized by a unique edge, and that $\ell^{a}_{\min}\neq\ell^{a'}_{\min}$
for all $a\neq a'$. Then for each $\omega\in\Omega$, the pair of
$m$-functions $m^{\pm}_{\omega}\left(z\right)$ uniquely determines
the sequence $\omega$.
\end{prop}

\begin{cor}
\label{cor:generic-BM-decorated}For a Baire-generic and Lebesgue
almost-sure choice of the decoration edge lengths, $m^{\pm}_{\omega}\left(z\right)$
uniquely determine $\omega$ for decorated $\Z$-graphs.
\end{cor}

The proof of Proposition~\ref{prop:BM-decorated} is completely analogous
to that of Theorem \ref{thm:BM}, and appears in the appendix of the
PhD thesis \cite[appendix A]{Sofer2026}.

\section{Dynamical characterization of $\sigma\left(H_{\Omega}\right)$ \label{sec:Kotani}}

In this section, we present several important spectral and dynamical
properties of tiling graphs that will be used in the proof of Theorem
\ref{thm:TMP}.

\subsection{Transfer matrices and the cocycle\label{subsec:Transfer-matrices}}

A basic tool which is required for our analysis is transfer matrices.
We give a brief review of those, and refer to \cite{damFil_book22}
for a more thorough overview.

As always, we fix the minimal subshift $\left(\Omega,\sft\right)$,
with an associated family of metric tiling graphs $\left(\Gamma_{\omega}\right)_{\omega\in\Omega}$.
Consider the following ODE on $\Gamma_{\omega}$:
\begin{equation}
-\frac{d^{2}}{dx^{2}}\phi_{\omega}=E\phi_{\omega}\,\,\,\,(E\in\mathbb{R}),\label{eq:ODE-transfer}
\end{equation}
with the Kirchhoff condition imposed at all vertices of $\Gamma_{\omega}$.
For a solution $\phi_{\omega}$ to (\ref{eq:ODE-transfer}), we denote
\begin{equation}
\Phi_{\omega}:=\left(\begin{array}{c}
\phi_{\omega}\left(o\right)\\
\sum_{{e\in\Gamma^{+}_{\omega}:\,e\sim o}}\left.\partial\phi_{\omega}\right|_{e}\left(o\right)
\end{array}\right)\in\C^{2},\label{eq:initial-ode}
\end{equation}
where the sum is taken over edges incident to the origin from the
right, and the derivatives are taken in the positive direction. Standard
ODE theory then shows that for all but a discrete set of $E\in\mathbb{R}$
(described after Definition \ref{def:Bad-energies}), there exists
a unique matrix $\mw\in SL\left(2,\mathbb{C}\right)$ such that for
any solution $\phi_{\omega}$ to (\ref{eq:ODE-transfer}),
\begin{equation}
\Phi_{\sft\omega}=\mw\Phi_{\omega}.\label{eq:trans-mat}
\end{equation}
The matrix $\mw$ is known as the one-step transfer matrix (or monodromy
matrix). One can similarly define the $n$-step transfer matrix by
\begin{equation}
\mnw:=\begin{cases}
\mtw{\sft^{n-1}}\cdot...\cdot\mw, & n\geq1,\\
I_{2}, & n=0,\\
\mtw{\sft^{n}}^{-1}\cdot...\cdot\mtw{\sft^{-1}}^{-1}, & n\leq-1.
\end{cases}\label{eq:trans-mat-n}
\end{equation}

The transfer matrices allow to define a dynamical system on $SL\left(2,\C\right)$
(known as a cocycle) through
\begin{equation}
\left(\omega,n\right)\mapsto\mnw.\label{eq:cocycle}
\end{equation}

\begin{rem}
Since neighboring tiles are attached to one another through a single
vertex, our cocycle consists of $2\times2$ matrices. If one allows
for more than a single attachment vertex, the cocycle will consist
of larger matrices, which is much more technically challenging. We
do not pursue this direction here, see for instance \cite{Becker}.
\end{rem}

\bigskip

There is a discrete set of values $E\in\R$ such that there exists
a nontrivial solution to (\ref{eq:ODE-transfer}) with initial condition
$\Phi_{\omega}=\vec{0}$ for some $\omega\in\Omega$ (see e.g., \cite[thm. 2.1 and cor. 2.4]{BanBerSmi_ahp12}).
Hence, the cocycle definition via (\ref{eq:trans-mat}) fails for
this set of energies (which are sometimes called resonances).
\begin{defn}
\label{def:Bad-energies}Denote by $\B\subset\R$ the set of energies
$E\in\R$ for which the cocycle $\mathcal{M}\left(\cdot,E\right)$
is not well-defined.
\end{defn}

In general, the set $\B$ may intersect $\sigma\left(H_{\Omega}\right)$
nontrivially -- for instance, this is the case for the aperiodic
antitrees studied in \cite{DamFanSuk_ahp20}. At such spectral points,
the analysis of $\sigma\left(H_{\Omega}\right)$ through the transfer
matrix applied below breaks down. However, since the set $\B$ is
discrete, it is enough to analyze $\sigma\left(H_{\Omega}\right)$
on $\R\backslash\B$ in order to prove that $\sigma\left(H_{\Omega}\right)$
is of Lebesgue measure zero. In Subsection \ref{subsec:Generic-Cantor-spectrum},
we show that generically $\B\cap\sigma\left(H_{\Omega}\right)=\varnothing$
for decorated $\Z$-graphs.

\bigskip

\subsection{Characterization of the spectrum via the Lyapunov exponent}

In this subsection, we provide a dynamical characterization of the
spectrum, by showing that $\sigma\left(H_{\omega}\right)$ is equal
(up to intersection with $\B$) to the zero set of the Lyapunov exponent
associated with the cocycle. Such a result has been proven in the
ergodic Schrödinger operator setting for many systems, including Sturmian
Hamiltonians and the Almost-Mathieu operator (cf. \cite{Damanik_PSPM_2007}).

Given $E\in\R\backslash\B$ and our cocycle $\mw$, Kingman's ergodic
theorem ensures that the quantity
\begin{equation}
L_{\omega}\left(E\right):=\lim_{n\rightarrow\infty}\frac{1}{n}\log\left\Vert \mnw\right\Vert ,\label{eq:Ly-exp}
\end{equation}
is identical for almost every $\omega\in\Omega$. We denote this almost-sure
value by $L_{\Omega}\left(E\right)$, and call $L_{\Omega}\left(E\right)$
the \textit{Lyapunov exponent} of the cocycle.
\begin{thm}
\label{thm:Kotani-eq} There exists a closed subset $\sigma\left(H_{\Omega}\right)\subset\R$
such that $\sigma\left(H_{\omega}\right)=\sigma\left(H_{\Omega}\right)$
for all $\omega\in\Omega$, and furthermore,
\begin{equation}
\sigma\left(H_{\Omega}\right)\backslash\B=\mathcal{Z}:=\set{E\in\R\backslash\B}{L_{\Omega}\left(E\right)=0}.\label{eq:sigma=00003Dz}
\end{equation}
\end{thm}

The proof of Theorem \ref{thm:Kotani-eq} appears at the end of this
subsection. To prove it, we first need to present several lemmas.
These lemmas are adaptations of existing techniques to quantum graphs.
Since in this subsection we are mostly adapting many of the techniques
presented in previous works to the quantum graph setting, some of
the results will not be proven completely, and the proofs will only
be outlined, with reference to the relevant works.
\begin{defn}
Define
\begin{align}
\UH & :=\set{E\in\R\backslash\B}{\exists\gamma>1,C>0;\left\Vert \mnw\right\Vert \geq C\gamma^{\left|n\right|},\forall\omega\in\Omega,\forall n\in\mathbb{Z}},\label{eq:uh}\\
\NUH & :=\set{E\in\R\backslash\B}{L_{\Omega}\left(E\right)>0}\backslash\UH.\label{eq:NUH}
\end{align}
It is easy to see that $\R\backslash\B=\UH\sqcup\NUH\sqcup\mathcal{Z}$,
and this partition into uniformly hyperbolic, non-uniformly hyperbolic
and non-hyperbolic energies forms the basis to our spectral analysis.
\end{defn}

\begin{lem}[{\cite[thm. 1.2]{MR3543643}}]
\label{fact:fact1}Let $E\in\UH$. Then there is no nontrivial solution
$\phi_{\omega}$ of (\ref{eq:ODE-transfer}) such that the vector
$\Phi_{n}$ in (\ref{eq:initial-ode}) is subexponentially bounded
in both directions in $\Gamma_{\omega}$, i.e. there is no nontrivial
solution such that for every $\varepsilon>0$ there exists $C_{\varepsilon}>0$
with
\begin{equation}
\|\Phi_{n}\|\le C_{\varepsilon}e^{\varepsilon|n|}\qquad\forall n\in\mathbb{Z}.\label{eq:subexponential}
\end{equation}
\end{lem}

Recall that a maximal spectral measure $\rho_{\omega}$ for $H_{\omega}$
is a Borel measure on $\R$ such that, for every Borel set $A\subset\mathbb{R}$,
\begin{equation}
\rho_{\omega}(A)=0\quad\Longleftrightarrow\quad\chi_{A}(H_{\omega})=0,\label{eq:max-spec}
\end{equation}
where $\chi_{A}(H_{\omega})$ denotes the spectral projection of $H_{\omega}$
associated with $A$. We say that a property holds for spectrally
almost every $E$ if it holds for $\rho_{\omega}$-almost every $E$,
and call $A$ spectrally null if $\rho_{\omega}(A)=0$.
\begin{lem}
\label{fact:fact2} For spectrally almost every $E\in\sigma\left(H_{\omega}\right)$,
there is a nontrivial solution $\phi_{\omega}$ of (\ref{eq:ODE-transfer})
such that for all $\beta>0$,
\begin{equation}
e^{-\beta\cdot\left|x\right|}\phi_{\omega}\left(x\right)\in L^{2}\left(\Gamma_{\omega}\right),\label{eq:L2-decay-2}
\end{equation}
where $\left|x\right|$ denotes the geodesic distance of $x$ from
the origin $o$. In particular, the vectors $\Phi_{n}$ in (\ref{eq:initial-ode})
are subexponentially bounded as in (\ref{eq:subexponential}).
\end{lem}

\begin{proof}[Proof of Lemma \ref{fact:fact2}]
 By \cite[thm. 1.1]{BouSto_jram03}, for the result to hold, it suffices
to verify two conditions:\\
1. For all $\alpha>0$ and $x_{0}\in\Gamma_{\omega}$,
\begin{equation}
\lim_{R\rightarrow\infty}e^{-\alpha R}\left|B^{\omega}_{R}\left(x_{0}\right)\right|=0,\label{eq:Brx}
\end{equation}
where $B^{\omega}_{R}\left(x_{0}\right)$ is the geodesic ball of
radius $R$ around $x_{0}\in\Gamma_{\omega}$ and $\left|\cdot\right|$
denotes the Lebesgue measure.\\
2. For all $\omega\in\Omega$ and $t>0$, the heat semigroup $e^{-tH_{\omega}}$
is ultracontractive, i.e.,
\begin{equation}
e^{-tH_{\omega}}:L^{2}\left(\Gamma_{\omega}\right)\rightarrow L^{\infty}\left(\Gamma_{\omega}\right).\label{eq:heat-semigroup}
\end{equation}

The first condition holds since the Lebesgue measure of geodesic balls
in $\Gamma_{\omega}$ grows piecewise linearly with $R$. The fact
that the second condition holds is shown in \cite[sec. 4.2]{Bifulco2023}.
\end{proof}

\begin{lem}
\label{lem:Subexponential}Suppose that $\phi_{\omega}$ is a subexponentially
bounded solution of (\ref{eq:ODE-transfer}) on $\Gamma_{\omega}$.
Meaning, for all $\kappa>0$ there is $C>0$ such that 
\begin{equation}
\left|\phi_{\omega}\left(x\right)\right|\leq Ce^{\kappa\left|x\right|},\quad\forall x\in\Gamma_{\omega}.\label{eq:subexp}
\end{equation}
Then $E\in\sigma\left(H_{\omega}\right)$.
\end{lem}

The proof is based on using the subexponential boundedness to construct
a Weyl sequence for $E$. The full details are given in \cite[lem. 4.13]{Sofer2026}.

\begin{proof}[Proof of Theorem \ref{thm:Kotani-eq}]
 The fact that $\sigma\left(H_{\omega}\right)$ is independent of
$\omega$ follows from standard arguments for minimal dynamical systems,
see for instance \cite[cor. 4.5]{Beckus2017}, \cite[prop. 2.1]{DamFilGor_ahp15},
\cite[thm. 5]{GruLenVes_jfa07} and references therein for similar
proofs. We wish to prove the equality (\ref{eq:sigma=00003Dz}). Recall
that $\mathbb{R}\backslash\B=\UH\sqcup\NUH\sqcup\mathcal{Z}$, and
further explore this decomposition as follows. Since the dynamical
system $\left(\Omega,\sft\right)$ satisfies Boshernitzan's condition,
we conclude by \cite[thm. 1]{MR2219271} that the Lyapunov exponent
$L_{\Omega}\left(E\right)$ in (\ref{eq:Ly-exp}) converges uniformly
in $\omega\in\Omega$. This implies that $\NUH=\varnothing$ (see
\cite[app. A]{DAMANIK_LENZ_2025}), and $\mathbb{R}\backslash\B=\UH\sqcup\mathcal{Z}$.
We thus only need to prove that $\sigma\left(H_{\Omega}\right)\backslash\B=\left(\mathbb{R}\backslash\B\right)\backslash\left(\UH\right)$,
and do it by showing both inclusions. Fix arbitrary $\omega\in\Omega$
and assume that $E_{0}\in\UH$. Since $\UH$ is open, we can choose
some open interval $I_{E_{0}}\subset\UH$ around $E_{0}$. For each
$E\in I_{E_{0}}$, Lemma \ref{fact:fact1} rules out a nontrivial
solution of (\ref{eq:ODE-transfer}) on $\Gamma_{\omega}$ whose vectors
$\Phi_{n}$ are subexponentially bounded in both directions. On the
other hand, by Lemma \ref{fact:fact2}, such a solution exists for
spectrally almost every $E\in\sigma\left(H_{\omega}\right)$. Hence
$I_{E_{0}}$ is spectrally null, that is $\chi_{I_{E_{0}}}\left(H_{\omega}\right)=0$.
Since $I_{E_{0}}$ is open and $\sigma\left(H_{\omega}\right)=\text{supp}\left(\rho_{\omega}\right)$,
it follows that $I_{E_{0}}\cap\sigma\left(H_{\omega}\right)=\varnothing$.
In particular $E_{0}\notin\sigma\left(H_{\omega}\right)$ and we get
that $\sigma\left(H_{\Omega}\right)\backslash\B=\sigma\left(H_{\omega}\right)\backslash\B\subset\left(\mathbb{R}\backslash\B\right)\backslash\left(\UH\right)$.

To prove the other inclusion, Lemma \ref{lem:Subexponential} tells
us that it suffices to show that for each $E\in\mathcal{Z}$, the
associated solution of (\ref{eq:ODE-transfer}) has subexponential
growth. But this is true by the definition of $\mathcal{Z}$ and the
uniform convergence of the Lyapunov exponent.
\end{proof}

\section{proof of Theorems \ref{thm:TMP} and \ref{thm:Cantor}\label{sec:Cantor spectrum proof}}

\subsection{The oracle theorem and proof of Theorem \ref{thm:TMP}}

We are now ready to prove Theorem \ref{thm:TMP} (we in fact prove
Theorem \ref{thm:Weaker-generalize-Cantor}, which immediately implies
Theorem \ref{thm:TMP}). For this, we will apply the following fundamental
result by Kotani \cite{Kotani1987}:
\begin{thm}
\label{thm:-Kotani} For almost every $\omega\in\Omega$ (with respect
to the measure $\mu$ from (\ref{eq:Boshernitzan})), the functions
$m^{\omega}_{\pm}$ are reflectionless on $\mathcal{Z}$. In other
words, for almost every $\omega\in\Omega$ and Lebesgue almost every
$E\in\mathcal{Z}$,
\begin{equation}
m^{-}_{\omega}\left(E+i0\right)=-\overline{m^{+}_{\omega}\left(E+i0\right)},\label{eq:reflectionless}
\end{equation}
where $m^{\pm}_{\omega}\left(E+i0\right)=\lim_{\epsilon\rightarrow0^{+}}m^{\pm}_{\omega}\left(E+i\epsilon\right)$.
\end{thm}

In other words, Kotani's theorem states that the set
\begin{equation}
\D\left(\mathcal{Z}\right):=\set{\omega\in\Omega}{\text{\ensuremath{m^{\pm}_{\omega}\left(z\right)} are reflectionless on }\mathcal{Z}},\label{eq:Dz}
\end{equation}
is a set of full measure. We note that while Kotani's original result
was established for ergodic potentials on $\R$, it can be directly
extended to tiling graphs by following essentially the same steps,
adapted to metric graphs (see \cite[sec. 2]{Kotani1987}).

For the sake of the next definition we introduce the notations $\N^{+}:=\left\{ 0,1,2,\ldots\right\} $
and $\N^{-}:=\left\{ -1,-2,-3,\ldots\right\} $.
\begin{defn}
Denote by $\mathcal{R}_{\pm}$ the restriction maps of a sequence
$\omega$ from $\mathbb{Z}$ to $\mathbb{N^{\pm}}$:
\begin{align}
 & \mathcal{R}_{\pm}:\A^{\Z}\rightarrow\A^{\N^{\pm}},\label{eq:R1}\\
 & \mathcal{R}_{\pm}\left(\left(\omega\left(n\right)\right)_{n\in\Z}\right)=\left(\omega\left(n\right)\right)_{n\in\N^{\pm}}.\label{eq:R2}
\end{align}
\end{defn}

\begin{thm}[Oracle theorem]
\label{thm:oracle} Suppose that $\mathcal{Z}$ has positive Lebesgue
measure. Then for the set $\D\left(\mathcal{Z}\right)$ in (\ref{eq:Dz}),
the mapping
\begin{align}
\mathcal{R} & :\mathcal{R}_{-}\left(\D\left(\mathcal{Z}\right)\right)\rightarrow\mathcal{R}_{+}\left(\D\left(\mathcal{Z}\right)\right),\label{eq:R3}\\
 & \mathcal{R}_{-}\left(\omega\right)\mapsto\mathcal{R}_{+}\left(\omega\right),\quad\omega\in\D\left(\mathcal{Z}\right),\label{eq:R4}
\end{align}
 is well-defined, bijective, and uniformly continuous. The same holds
for $\mathcal{R}^{-1}$.
\end{thm}

In a nutshell, the oracle theorem states that if $\left|\mathcal{Z}\right|>0$,
then the set $\D\left(\mathcal{Z}\right)$ is in some sense ``deterministic'',
and one can predict one half of the sequence $\omega\in\D\left(\mathcal{Z}\right)$
by knowing the other (hence the name `oracle'). This is analogous
to Remling's Oracle Theorem for one-dimensional Schrödinger and Jacobi
operators \cite{MR2811596}. Our result is also closely related to
\cite[thm. 3.7]{DamFanSuk_ahp20}.
\begin{proof}
By the Borg-Marchenko Theorem \ref{thm:BM}, we know that the restrictions
$\mathcal{R}_{\pm}\left(\omega\right)$ determine $m^{\pm}_{\omega}$
uniquely, and this correspondence is continuous by Corollary \ref{cor:loc-uni}.
Moreover, since we assume that $m^{-}_{\omega}\left(E+i0\right)=-\overline{m^{+}_{\omega}\left(E+i0\right)}$
on the positive measure set $\mathcal{Z}$, and Herglotz functions
are uniquely determined by their boundary values on a set of positive
measure, we obtain a 1-1 correspondence between the functions $m^{\pm}_{\omega}\left(z\right)$
for $\omega\in\D\left(\mathcal{Z}\right)$. We thus see that $\mathcal{R}$
is well-defined, bijective, and continuous by the chain of continuous
bijections
\begin{equation}
\mathcal{R}_{+}\left(\omega\right)\iff m^{+}_{\omega}\left(z\right)\iff m^{-}_{\omega}\left(z\right)\iff\mathcal{R}_{-}\left(\omega\right).\label{eq:M-R-Correspondence}
\end{equation}
Due to \cite[lem. 5]{kotani1985support} and Corollary \ref{cor:loc-uni},
we also know that $\D\left(\mathcal{Z}\right)$ is closed, and thus
compact. Since $\mathcal{R}_{\pm}$ are continuous, $\mathcal{R}_{\pm}\left(\D\left(\mathcal{Z}\right)\right)$
are compact as well. Thus the continuous bijection $\mathcal{R}:\mathcal{R}_{-}\left(\D\left(\mathcal{Z}\right)\right)\to\mathcal{R}_{+}\left(\D\left(\mathcal{Z}\right)\right)$
is a homeomorphism. Consequently, both $\mathcal{R}$ and $\mathcal{R}^{-1}$
are uniformly continuous.
\end{proof}

The following lemma will allow us to apply the oracle theorem simultaneously
for entire orbits in $\Omega$:
\begin{lem}
\label{lem:shift-invariant}The set $\D\left(\mathcal{Z}\right)$
is shift-invariant.
\end{lem}

\begin{proof}
Let $z\in\C_{+}$. For $\omega\in\Omega$ fixed, we denote the elements
of the transfer matrix $\mw$ by $\mathcal{M}_{ij}$. Then, combining
the definition (\ref{eq:m3}) for $m^{\pm}_{\omega}\left(z\right)$
with definition (\ref{eq:trans-mat}) for the transfer matrix (and
carefully keeping notice of the orientation of the derivative with
respect to the origin), direct computation shows that,
\begin{align}
 & m^{-}_{\sft\omega}\left(z\right)=\frac{\sum_{e\sim o\left(\Gamma^{-}_{\sft\omega}\right)}\left.\frac{d}{dx}f_{\sft\omega}\right|_{e}\left(o\left(\Gamma^{-}_{\sft\omega}\right);z\right)}{f_{\sft\omega}\left(o\left(\Gamma^{-}_{\sft\omega}\right);z\right)}=-\frac{\sum_{e\sim o\left(\Gamma^{-}_{\sft\omega}\right)}\left.\frac{d}{dx}f_{\sft\omega}\right|_{e}\left(o\left(\Gamma^{+}_{\sft\omega}\right);z\right)}{f_{\sft\omega}\left(o\left(\Gamma^{+}_{\sft\omega}\right);z\right)}\nonumber \\
= & -\frac{\mathcal{M}_{21}f_{\omega}\left(o\left(\Gamma^{+}_{\sft\omega}\right);z\right)+\mathcal{M}_{22}\sum_{e\sim o\left(\Gamma^{+}_{\omega}\right)}\left.\frac{d}{dx}f_{\omega}\right|_{e}\left(o\left(\Gamma^{+}_{\omega}\right);z\right)}{\mathcal{M}_{11}f_{\omega}\left(o\left(\Gamma^{+}_{\sft\omega}\right);z\right)+\mathcal{M}_{12}\sum_{e\sim o\left(\Gamma^{+}_{\omega}\right)}\left.\frac{d}{dx}f_{\omega}\right|_{e}\left(o\left(\Gamma^{+}_{\omega}\right);z\right)}\nonumber \\
= & -\frac{\mathcal{M}_{21}f_{\omega}\left(o\left(\Gamma^{-}_{\sft\omega}\right);z\right)-\mathcal{M}_{22}\sum_{e\sim o\left(\Gamma^{-}_{\omega}\right)}\left.\frac{d}{dx}f_{\omega}\right|_{e}\left(o\left(\Gamma^{-}_{\omega}\right);z\right)}{\mathcal{M}_{11}f_{\omega}\left(o\left(\Gamma^{-}_{\sft\omega}\right);z\right)-\mathcal{M}_{12}\sum_{e\sim o\left(\Gamma^{-}_{\omega}\right)}\left.\frac{d}{dx}f_{\omega}\right|_{e}\left(o\left(\Gamma^{-}_{\omega}\right);z\right)}\nonumber \\
= & -\frac{\mathcal{M}_{21}-\mathcal{M}_{22}m^{-}_{\omega}\left(z\right)}{\mathcal{M}_{11}-\mathcal{M}_{12}m^{-}_{\omega}\left(z\right)},\label{eq:mTw-1}
\end{align}
where we note that $\left.f_{\omega}\right|_{e}\left(o\left(\Gamma^{-}_{\omega}\right);z\right)\neq0$,
otherwise $m^{-}_{\omega}$ would have a pole at $z$ (which cannot
happen for $z\in\C_{+}$). Similarly (and noting that the value of
$f_{\omega}\left(v;z\right)$ does not depend on the orientation with
respect to $v$), one can show that
\begin{align}
m^{+}_{\sft\omega}\left(z\right)= & \frac{\mathcal{M}_{21}+\mathcal{M}_{22}m^{+}_{\omega}\left(z\right)}{\mathcal{M}_{11}+\mathcal{M}_{12}m^{+}_{\omega}\left(z\right)}.\label{eq:mTw}
\end{align}

Now, let $\omega\in\D\left(\mathcal{Z}\right)$, i.e., there exists
$\tilde{\mathcal{Z}}\subset\mathcal{Z}$ of full measure such that
$m^{-}_{\omega}\left(E+i0\right)=-\overline{m^{+}_{\omega}\left(E+i0\right)}$
for all $E\in\tilde{\mathcal{Z}}$. Taking $z=E+i\epsilon$ with $\epsilon\rightarrow0$
in (\ref{eq:mTw}) above, one obtains that for all $E\in\tilde{\mathcal{Z}}$,
\begin{align}
 & -\overline{m^{+}_{\sft\omega}\left(E+i0\right)}\nonumber \\
= & -\frac{\mathcal{M}_{21}\left(\omega,E+i0\right)+\mathcal{M}_{22}\left(\omega,E+i0\right)\overline{m^{+}_{\omega}\left(E+i0\right)}}{\mathcal{M}_{11}\left(\omega,E+i0\right)+\mathcal{M}_{12}\left(\omega,E+i0\right)\overline{m^{+}_{\omega}\left(E+i0\right)}}\nonumber \\
= & -\frac{\mathcal{M}_{21}\left(\omega,E+i0\right)-\mathcal{M}_{22}\left(\omega,E+i0\right)m^{-}_{\omega}\left(E+i0\right)}{\mathcal{M}_{11}\left(\omega,E+i0\right)-\mathcal{M}_{12}\left(\omega,E+i0\right)m^{-}_{\omega}\left(E+i0\right)}\nonumber \\
= & m^{-}_{\sft\omega}\left(E+i0\right),\label{eq:mp-mm-equality}
\end{align}
where we have used the fact that $\mathcal{M}\left(\omega,E+i0\right)$
is a real-valued matrix, as the ODE (\ref{eq:ODE-transfer}) has real-valued
solutions at the limit $\epsilon\rightarrow0$. This completes the
proof.
\end{proof}

\begin{thm}
\label{thm:Weaker-generalize-Cantor}$\sigma\left(H_{\Omega}\right)$
has Lebesgue measure zero. Moreover, all isolated points of $\sigma\left(H_{\Omega}\right)$
(if any exist) are eigenvalues from the discrete set $\B\cap\sigma\left(H_{\Omega}\right)$.
In particular, up to a possible discrete set of eigenvalues, $\sigma\left(H_{\Omega}\right)$
is a generalized Cantor set.
\end{thm}

\begin{rem}
\label{rem:thm1.7.4.5}This is exactly Theorem \ref{thm:TMP}, where
the set $F$ is the set of isolated points in $\B\cap\sigma\left(H_{\Omega}\right)$.
\end{rem}

\begin{proof}
We know by Theorem \ref{thm:Kotani-eq} that $\sigma\left(H_{\omega}\right)\backslash\B=\sigma\left(H_{\Omega}\right)\backslash\B=\mathcal{Z}$
for all $\omega\in\Omega$. Since $\B$ is discrete, we can simply
show that $\left|\mathcal{Z}\right|=0$ to deduce that $\left|\sigma\left(H_{\Omega}\right)\right|=0$.
Assume by contradiction that this is not the case. Since $\left|\mathcal{Z}\right|>0$,
the oracle theorem gives that the mapping
\begin{equation}
\mathcal{R}:\left(\omega\left(n\right)\right)_{n\in\N^{-}}\mapsto\left(\omega\left(n\right)\right)_{n\in\N^{+}},\quad\omega\in\D\left(\mathcal{Z}\right),\label{eq:shift-invariant}
\end{equation}
is well-defined and uniformly continuous. Moreover, Kotani's theorem
\ref{thm:-Kotani} implies that the set $\D\left(\mathcal{Z}\right)\subset\Omega$
is of full measure.

By uniform continuity of $\mathcal{R}$ there exists $k=k\left(\mathcal{A}\right)\in\mathbb{N}$
(independent of $\omega$) such that for all $\omega\in\D\left(\mathcal{Z}\right)$,
the collection $\left(\omega\left(-k\right),...,\omega\left(-1\right)\right)$
uniquely determines $\omega\left(0\right)$. Using the shift invariance
of $\mathcal{D}(\mathcal{Z})$ (Lemma \ref{lem:shift-invariant}),
we repeat this argument for all positive shifts $\left\{ \sft^{n}\omega\right\} _{n\in\N}$,
deducing that the collection $\left(\omega\left(-k\right),...,\omega\left(-1\right)\right)$
uniquely determines $\left(\omega\left(n\right)\right)_{n\in\N^{+}}$.
Similarly, using the uniform continuity of $\mathcal{R}^{-1}$ together
with the shift invariance of $\mathcal{D}(\mathcal{Z})$ and applying
negative shifts, we get that $\left(\omega\left(n\right)\right)_{n\in\N^{+}}$
uniquely determines $\left(\omega\left(n\right)\right)_{n\in\N^{-}}$.
Summing up, we deduce that the collection $\left(\omega\left(-k\right),...,\omega\left(-1\right)\right)$
determines the entire sequence $\omega$. Thus,
\begin{equation}
\#\D\left(\mathcal{Z}\right)\leq\left|\mathcal{A}\right|^{k}<\infty.\label{eq:shift-finite}
\end{equation}
 We conclude that every sequence $\omega\in\D\left(\mathcal{Z}\right)$
has only finitely many possible shifts, and in particular it is periodic
-- contradicting the aperiodicity of $\Omega$. This gives that $\left|\mathcal{Z}\right|=\left|\sigma\left(H_{\Omega}\right)\backslash\B\right|=\left|\sigma\left(H_{\Omega}\right)\right|=0$.

In particular, $\sigma\left(H_{\Omega}\right)$ contains no intervals.
Since it is also closed, then it is nowhere dense in $\R$. Moreover,
standard theory \cite[thm. 4.2.4]{damFil_book22} implies the discrete
spectrum $\sigma_{\text{disc}}\left(H_{\Omega}\right)\backslash\B$
is empty. Hence, $\sigma\left(H_{\Omega}\right)$ is closed, nowhere
dense, and contains no isolated points, up to possible isolated points
in $\B\cap\sigma\left(H_{\Omega}\right)$ (which are automatically
eigenvalues as isolated spectral points), completing the proof.
\end{proof}

\begin{rem}
We emphasize that energies in $\B\cap\sigma\left(H_{\Omega}\right)$
may correspond either to isolated eigenvalues or to embedded eigenvalues.
In the latter case, these energies are accumulation points of the
spectrum and therefore do not affect its Cantor nature. The question
of whether attaching graph decorations or introducing internal vertex
structures necessarily leads to the opening of spectral gaps (and
hence to isolated eigenvalues) has been studied in various settings,
see for instance \cite{kuchment2005quantum,AvrExnLas_prl94,do2017resonant}.
These works identify conditions under which such local structures
produce genuine spectral gaps (and not embedded eigenvalues).
\end{rem}

\subsection{Generic Cantor spectrum for decorated $\protect\Z$-graphs -- proof
of Theorem~\ref{thm:Cantor}\label{subsec:Generic-Cantor-spectrum}}

By Theorem \ref{thm:TMP}, $\sigma\left(H_{\Omega}\right)\backslash\B$
is a generalized Cantor set. We have also seen that $\B\cap\sigma\left(H_{\Omega}\right)$
is discrete and now wish to examine the cases for which $\sigma\left(H_{\Omega}\right)\cap\B=\varnothing$,
so that the actual spectrum $\sigma\left(H_{\Omega}\right)$ is a
generalized Cantor set. In this subsection we prove that for the class
of decorated $\Z$-graphs (Definition \ref{def: decorated lines}),
$\sigma\left(H_{\Omega}\right)\cap\B=\varnothing$ holds Baire-generically
for the edge length set, which forms part of Theorem \ref{thm:Cantor}.
We conjecture that a similar genericity result should hold for a broader
class within tiling graphs, leading to generic Cantor spectrum, but
the proof below uses the decorated $\Z$-graph geometry in an essential
way.

For the remainder of this subsection, we fix a decorated $\Z$-graph
$\Gamma_{\omega}$ with no loops. As previously noted in Subsection~\ref{subsec:Transfer-matrices},
the discrete set $\B\cap\sigma\left(H_{\omega}\right)$ consists of
points $E\in\sigma\left(H_{\omega}\right)$ admitting a nontrivial
solution to (\ref{eq:ODE-transfer}) with initial condition $\Phi_{\omega}=\vec{0}$.
Consider the generalized solution to the eigenvalue equation $H_{\omega}f=Ef$,
with $E\in\B\cap\sigma\left(H_{\omega}\right)$. Since for decorated
$\Z$-graphs the vector $\Phi_{\omega}=\vec{0}$ is evaluated at the
center of an edge, the ODE uniqueness principle on intervals implies
that $f$ vanishes at the horizontal infinite line graph. This also
means that $f$ can be chosen to be a compactly supported eigenfunction.
\begin{lem}
\label{lem:compact}Assume that $\left(E,f\right)$ is an eigenpair
of a decorated $\Z$-graph $\Gamma_{\omega}$ with $f$ compactly
supported. Then there exists a pair of decorations $\D_{1},\D_{2}$
(possibly $\D_{1}=\D_{2}$) such that for $i\in\{1,2\}$, $\left(E,f|_{\D_{i}}\right)$
is an eigenpair of the Laplacian $H|_{\D_{i}}$, with a Dirichlet
condition imposed at the base vertex of the decoration $\D_{i}$ and
the Kirchhoff condition everywhere else. Furthermore, $f$ vanishes
on all edges of $\Gamma_{\omega}$ which are either to the left of
$\D_{1}$ or to the right of $\D_{2}$.
\end{lem}

\begin{proof}
Let $\left(E,f\right)$ be as above. Since the line graph itself does
not admit compactly supported eigenfunctions, $f$ must be nonzero
at some decoration $\D$. Choose $\D_{1}$ to be the left-most decoration
of the comb such that this happens. Since $f\equiv0$ to the left
of $\D_{1}$, continuity then implies that $f\left(v_{a}\right)=0$
at the base vertex $v_{a}$ of the decoration $\D_{1}$. Now, by construction,
$f|_{\D_{1}}$ is a solution to the eigenvalue equation $-f''=Ef$,
satisfies the Dirichlet condition at the base of $\D_{1}$, and the
Kirchhoff condition everywhere else, as required. Similarly we choose
$\D_{2}$ to be the right-most decoration of the comb at which $f$
is non-zero. We get that the assertion of the lemma is satisfied for
$\D_{1}$ and $\D_{2}$. It might also be that $\D_{1}=\D_{2}$.
\end{proof}

Thanks to Theorem~\ref{thm:Kotani-eq}, there exists a closed subset
$\sigma\left(H_{\Omega}\right)\subset\R$ such that $\sigma\left(H_{\omega}\right)=\sigma\left(H_{\Omega}\right)$
for all $\omega\in\Omega$. We now prove
\begin{thm}
\label{thm:Generic-assumption}There exists a Baire-generic set $A$
of edge lengths for which $\B\cap\sigma\left(H_{\Omega}\right)=\varnothing$
for all $\vec{\ell}\in A$.
\end{thm}

\begin{rem}
In light of the discussion above, Theorem~\ref{thm:Generic-assumption}
states that Baire-generically (over the edge length set) $\sigma\left(H_{\Omega}\right)$
admits no compactly supported eigenfunctions. We note that a generic
exclusion of compactly supported eigenfunctions was proven for periodic
discrete graphs in \cite{Faust2025,M.Faust2025} (where it is also
referred to as flat bands or eigenvalues of infinite multiplicity).
\end{rem}

\begin{proof}
By the discussion preceding Lemma~\ref{lem:compact}, it suffices
to construct a Baire-generic set of edge lengths such that the corresponding
decorated $\Z$-graphs do not admit compactly supported eigenfunctions.
For a compactly supported eigenfunction $f$ we may apply Lemma~\ref{lem:compact}
and get the existence of decorations $\D_{1},\D_{2}$ as in the lemma.
Denote by $\left[\D_{1},\D_{2}\right]$ the subgraph of $\Gamma_{\omega}$
which lies between $\D_{1},\D_{2}$, including these two decorations.
Denote by $H_{\left[\D_{1},\D_{2}\right]}$ the Laplacian on $\left[\D_{1},\D_{2}\right]$
with Kirchhoff conditions at all vertices. We get that $f$ is an
eigenfunction of $H_{\left[\D_{1},\D_{2}\right]}$ and in addition
vanishes at the base vertices of $\D_{1}$ and $\D_{2}$. Since the
decorations are assumed to contain no loops, applying \cite[thm. 3.6]{BerLiu_jmaa17}
(see also \cite[thm. 3.2]{alon2024generic}) yields that there exists
a Baire-generic set of edge lengths such that no eigenfunction of
$H_{\left[\D_{1},\D_{2}\right]}$ vanishes at vertices. Hence, such
a choice of edge lengths excludes the possibility of an eigenfunction
$f$ as above which is supported on $\left[\D_{1},\D_{2}\right]$.
Now, we apply this argument for all possible (countably many) subgraphs
of $\Gamma_{\omega}$ of the type $\left[\D_{3},\D_{4}\right]$, i.e.,
subgraphs which are bounded between two arbitrary decorations. Take
the intersection of all Baire-generic sets of edge lengths which by
\cite[thm. 3.6]{BerLiu_jmaa17} ensure that no eigenfunction of the
Kirchhoff Laplacian on any of these subgraphs vanishes at a vertex.
This (countable) intersection is a Baire-generic set of edge lengths
which (by Lemma~\ref{lem:compact}) permits no compactly supported
eigenfunction.
\end{proof}

The proof of Theorem \ref{thm:Cantor} now follows immediately:
\begin{proof}[Proof of Theorem \ref{thm:Cantor}]
Theorem \ref{thm:Weaker-generalize-Cantor} for decorated $\Z$-graphs
can be proven by repeating verbatim the proof of Theorem \ref{thm:TMP},
using the adapted Borg-Marchenko result from Proposition \ref{prop:BM-decorated},
under the generic condition on decoration edge lengths assumed in
the theorem (see also Corollary \ref{cor:generic-BM-decorated}).
In this case, we immediately get that $\sigma\left(H_{\Omega}\right)$
is of zero Lebesgue measure. Furthermore, by Theorem \ref{thm:Generic-assumption},
we have that $\B\cap\sigma\left(H_{\Omega}\right)=\varnothing$ under
a Baire-generic choice of the edge lengths (and since the intersection
of generic sets is generic, we can also ensure that the conditions
of Proposition \ref{prop:BM-decorated} still hold). For such generic
choices of the edge lengths, Theorem \ref{thm:Weaker-generalize-Cantor}
then also implies that $\sigma\left(H_{\Omega}\right)$ is a Cantor
set.
\end{proof}

\section{Detailed example -- Sturmian combs\label{sec:Discussion-and-Example}}

To illustrate the results developed in this paper, we now turn to
the concrete example of Sturmian comb graphs. These are the graphs
presented in Example \ref{exa: Sturm-comb} (cf. Figure \ref{fig: TilingGraphs}),
which consist of two possible tiles -- a trivial tile (a horizontal
line) and a nontrivial tile (a star graph with three edges). Following
the notation of the Sturmian sequences, (\ref{eq:sturm}), we denote
the corresponding operators by $\left(H_{\alpha,\theta}\right)^{\theta\in[0,1)}_{\alpha\in[0,1)}$.
By Theorem \ref{thm:TMP}, $\sigma\left(H_{\alpha,\theta}\right)$
is independent of $\theta$, and we thus focus on $H_{\alpha,0}=:H_{\alpha}$.
Concretely, let $\Gamma_{\alpha}$ denote the chain graph on $\Z$
with an attached dangling edge of length $\ell>0$ for each $n\in\Z$
such that $\omega_{\alpha}\left(n\right)=1$, where
\begin{equation}
\omega_{\alpha}\left(n\right)=\chi_{(1-\alpha,1]}\left(n\alpha\text{ mod \ensuremath{1}}\right).\label{eq:Sturmian-0}
\end{equation}

Given $\alpha\notin\Q$, we can define a family of periodic approximations
for $H_{\alpha}$ using the (infinite) continued fraction expansion
for $\alpha$:
\begin{equation}
\alpha=\frac{1}{a_{1}+\frac{1}{a_{2}+...}}.\label{eq:cont-frac}
\end{equation}
By truncating the continued fraction expansion at the $n$-th step,
we obtain a rational number $\alpha_{n}:=\frac{p_{n}}{q_{n}}$, which
through (\ref{eq:Sturmian-0}) corresponds to a periodic quantum graph
operator $H_{\alpha_{n}}$ (see Figure \ref{fig: periodic-approx}).

\begin{figure}[h]
\includegraphics[scale=0.5]{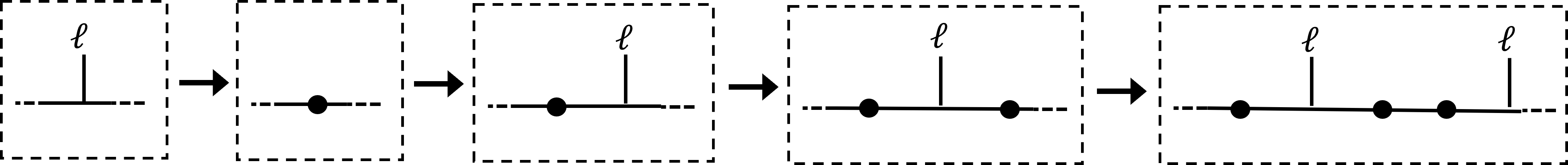}

\caption{Fundamental domains for the sequence of periodic approximations for
the Fibonacci comb with $\alpha=\frac{\sqrt{5}-1}{2}$. The $n$th
fundamental domain contains $F_{n}$ tiles, where $F_{n}$ is the
$n$th Fibonacci number. \label{fig: periodic-approx}}
\end{figure}

To analyze the spectra, we compute the one-step transfer matrix of
the possible tiles. They are:
\begin{equation}
\mathcal{M}_{a}\left(E=k^{2}\right)=\left(\begin{array}{cc}
\cos\left(kL\right) & \frac{1}{k}\sin\left(kL\right)\\
-\frac{1}{k}\sin\left(kL\right) & \cos\left(kL\right)
\end{array}\right),\label{eq:fibo-scatter-2}
\end{equation}
for the trivial tile, and
\begin{align}
 & \mathcal{M}_{b}\left(E=k^{2}\right)\nonumber \\
 & =\left(\begin{array}{cc}
\cos\left(2kL\right)-\frac{1}{2}\sin\left(2kL\right)\tan\left(k\ell\right) & -k\sin\left(2kL\right)-k\cos\left(kL\right)^{2}\tan\left(k\ell\right)\\
\frac{1}{k}\sin\left(2kL\right)-\frac{1}{k}\sin^{2}\left(kL\right)\tan\left(k\ell\right) & \cos\left(2kL\right)-\frac{1}{2}\sin\left(2kL\right)\tan\left(k\ell\right)
\end{array}\right),\label{eq:Fibo-scatter}
\end{align}
for the nontrivial tile. Note that the cocycle $\mathcal{M}\left(k^{2}\right)$
is well-defined for all $k$, except for a discrete subset of energies,
as in Definition \ref{def:Bad-energies}:
\begin{equation}
\B_{\ell}:=\left\{ \left(\frac{\pi/2+\pi m}{\ell}\right)^{2}:m\in\Z\right\} .\label{eq:B-comb}
\end{equation}
To make our analysis more convenient, we assume from now on that these
``forbidden'' values of $k^{2}$ are excluded from the spectrum,
similar to the Baire-generic assumption in Theorem \ref{thm:Generic-assumption}:
\begin{assumption}
\label{assu:bad-energies}$\sigma\left(H_{\alpha}\right)\cap\B_{\ell}=\varnothing$.
\end{assumption}

Under this generic assumption, Theorem \ref{thm:TMP} ensures that
$\sigma\left(H_{\alpha}\right)$ is of zero Lebesgue measure. In addition,
the standard Gordon lemma (see, e.g., \cite[prop. 2]{Suetoe_cmp87})
also guarantees that $\sigma\left(H_{\alpha}\right)$ contains no
eigenvalues, and so the spectral measures are purely singular continuous
and the spectrum is a generalized Cantor set.

On the other hand, if Assumption~\ref{assu:bad-energies} is not
satisfied, then one can show that $\sigma\left(H_{\alpha}\right)$
will in fact contain isolated eigenvalues of infinite multiplicity
at the spectral points where $\mw$ does not exist (and in particular,
will not be a Cantor set). In this case, the integrated density of
states is discontinuous at these spectral points, similarly to the
case of the Penrose and Ammann-Beenker tilings studied in \cite{DamEmbFilMei_exp23}.
In our parallel work, \cite{BanSof_GLT2026} (see also \cite{Bandb}),
about the IDS gap labels we also examine the Sturmian comb example.
We completely characterize there for which energies and tiles the
infinite multiplicity eigenvalues appear, and compute the associated
jump discontinuities of the IDS explicitly.

We return to the periodic approximants which are defined via the rational
approximants $\alpha_{n}=\frac{p_{n}}{q_{n}}$, with the corresponding
graphs $\Gamma_{\alpha_{n}}$and operators $H_{\alpha_{n}}$ (as in
Figure~\ref{fig: periodic-approx}). With some careful analysis of
the corresponding transfer matrices and applying the methods presented
in \cite{Bellissard1989,Suetoe_cmp87} and the result in \cite[thm. 2.5.1]{Beckus2016a},
we get the following:
\begin{thm}
\label{thm:Per-approx}Denoting $\Sigma_{\alpha_{n}}:=\sigma\left(H_{\alpha_{n}}\right)\cup\sigma\left(H_{\alpha_{n+1}}\right)$,
we have 
\begin{equation}
\sigma\left(H_{\alpha}\right)=\bigcap_{n\in\N}\Sigma_{\alpha_{n}}\label{eq:approx1}
\end{equation}
and 
\begin{equation}
\Sigma_{\alpha_{n}}\rightarrow\sigma\left(H_{\alpha}\right)\label{eq:Hausdorff-conv}
\end{equation}
in the Hausdorff topology.
\end{thm}

The result is interesting in particular as the sets $\Sigma_{\alpha_{n}}$
are not compact (though closed). A detailed analysis and proof appear
in the PhD thesis \cite[theorems 7.3, 7.4]{Sofer2026}. Figure \ref{fig: periodic-approx-2}
demonstrates how the sequence $\Sigma_{\alpha_{n}}$ converges to
the limiting set $\sigma\left(H_{\alpha}\right)$. It also sheds light
on the Cantor set nature of the spectrum, and to it being of zero
measure. Figures \ref{fig: butterfly-1} and \ref{fig: butterfly}
display the fractal nature of the spectrum $\sigma\left(H_{\alpha}\right)$
as a function of $\alpha$; this can be thought of as a version of
the Kohmoto butterfly (cf. \cite{band2024dry,KohOon_pla84,KohKadTan_prl83})
for quantum graphs. We also refer to \cite{baradaran2023magnetic},
where a plot similar to \ref{fig: butterfly} was produced for a Harper-like
quantum graph model, showing resemblance to the Hofstadter butterfly
\cite{Hof_prb76}.
\begin{figure}
\includegraphics[scale=0.55]{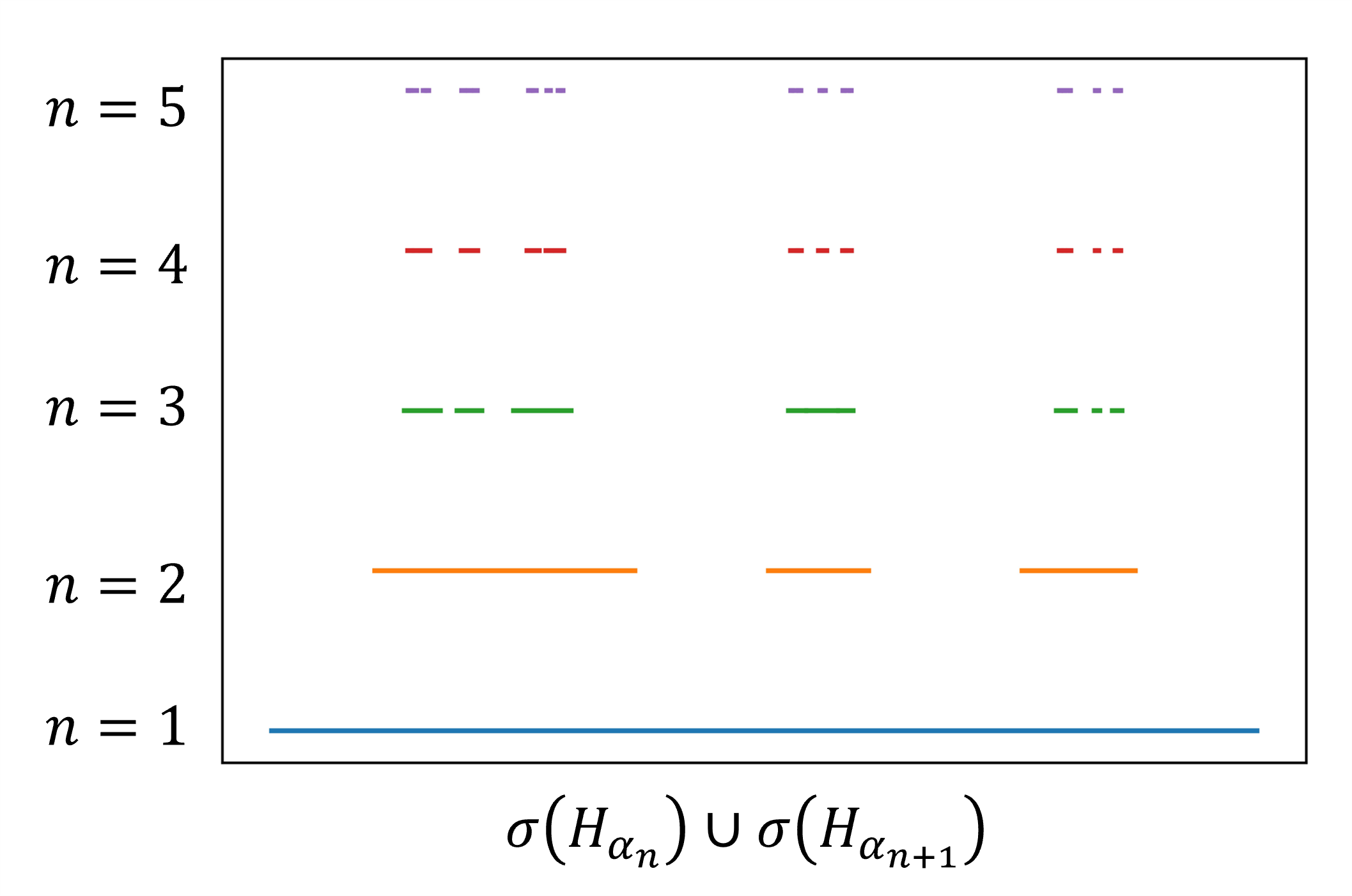}

\caption{Parts of the sequence of periodic approximations $\sigma\left(H_{\alpha_{n}}\right)\cup\sigma\left(H_{\alpha_{n+1}}\right)$
for the spectrum $\sigma\left(H_{\alpha}\right)$ of the Fibonacci
comb. \label{fig: periodic-approx-2}}
\end{figure}

\begin{figure}
\includegraphics[scale=0.75]{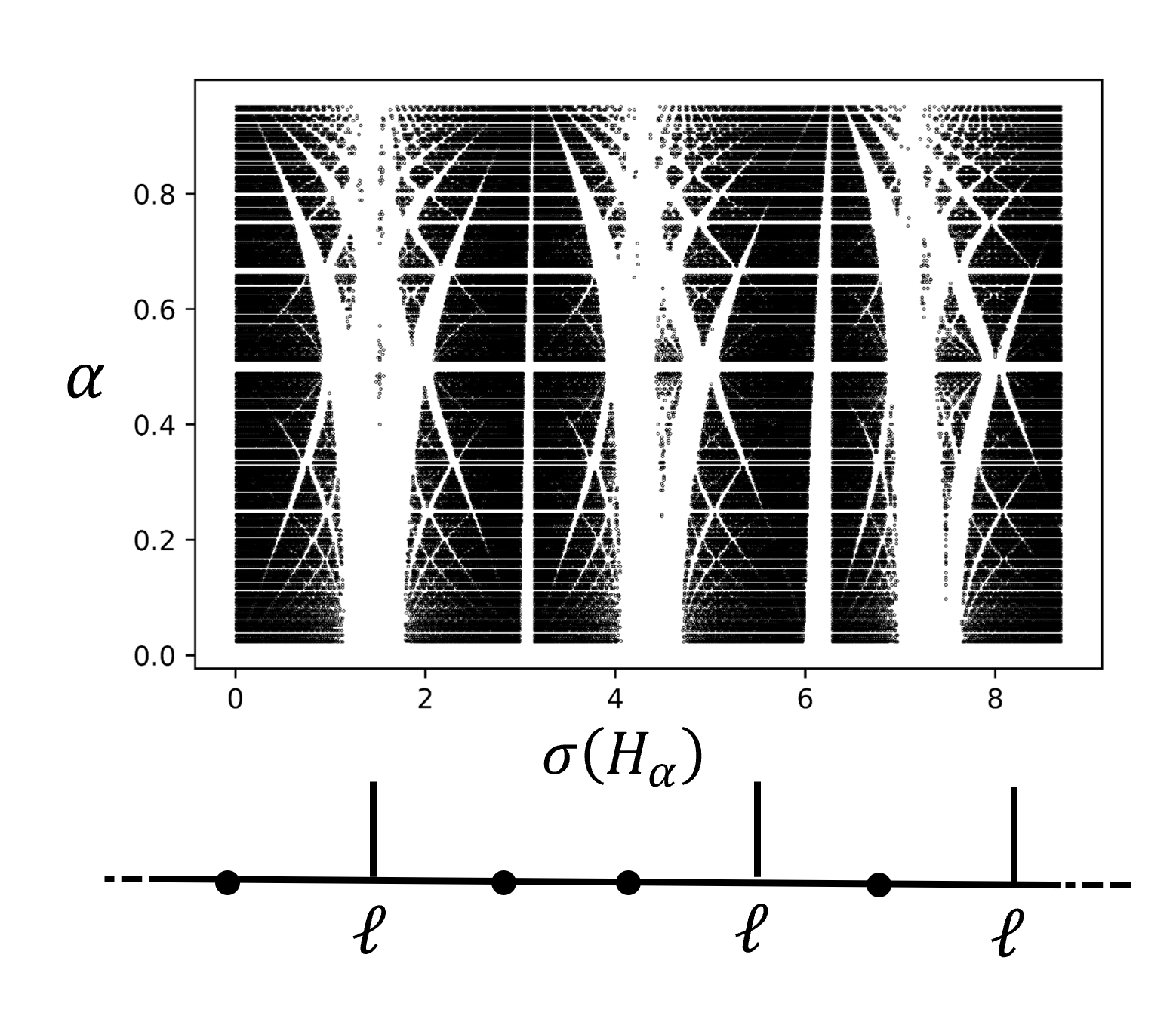}

\caption{The \textquotedblleft Kohmoto butterfly\textquotedblright{} for Sturmian
combs, showing the fractal structure of $\sigma\left(H_{\alpha}\right)$
changing gradually as the energy increases. \label{fig: butterfly}}
\end{figure}

\bibliographystyle{alpha}
\bibliography{GlobalBib_240308}

\end{document}